\documentclass{siamart251216}
\allowdisplaybreaks
\usepackage{amssymb}
\usepackage{graphics,epsfig}
\usepackage{todonotes}
\usepackage{hyperref}
\usepackage[scr=boondox]{mathalpha}
\usepackage{pgfplots}
\pgfplotsset{compat=1.18}

\providecommand{\dist}{\operatorname{dist}}
\providecommand{\diam}{\operatorname{diam}}
\providecommand{\patch}[1]{\omega_{#1}}
\newcommand{\matId}{\underline{\mathbf{I}}}
\providecommand{\Id}{\operatorname{I}}
\providecommand{\cl}[1]{\overline{#1}}
\providecommand{\Fbd}{\mathcal{F}_h^{\partial}}
\providecommand{\Ebd}{\mathcal{E}_h^{\partial}}

\DeclareMathOperator\Grad{\mathbf{grad}}
\DeclareMathOperator\Jac{\underline{\mathbf{D}}}

\newcommand\Ld{L^2(\Omega)}
\newcommand\Hu{H^1(\Omega)}
\newcommand\Huh{H^1(\Oh)}
\newcommand\Huo{H^1_0(\Omega)}
\newcommand\Huoh{H^1_0(\Omega_h)}
\newcommand\RE{\mathbb{R}}
\newcommand\Th{\mathcal{T}_h}
\newcommand\ThB{\mathcal{T}_h^B}
\newcommand\OB{\Omega_h^B}
\newcommand\OBpatch{\Omega_{h,\delta}}
\newcommand\pO{{\partial\Omega}}
\newcommand\Oh{{\Omega_h}}
\newcommand\OOh{\Omega\cap\Oh}
\newcommand\pOh{{\partial\Oh}}
\newcommand\Obar{\Omega\cup\Oh}
\newcommand\OmOh{\Omega\setminus\Oh}
\newcommand\OhmO{\Oh\setminus\Omega}

\newcommand\vit{\tilde v^I}
\newcommand\vi{v^I}
\newcommand\Tt{{\widetilde T}}
\newcommand\Spt{S^+_T}
\newcommand\Smt{S^-_T}
\newcommand\extension{\mathcal E}
\newcommand\R{\mathcal R}

\newcommand\lunette{\mathcal{L}_h}
\newcommand\lunetteconv{\lunette^+}
\newcommand\lunetteconc{\lunette^-}

\newcommand\lenoir[1]{#1^\#}
\newcommand\bfx{\mathbf{x}}
\newcommand\resto{\mathrm{R}}
\newcommand\dx{d\bfx}
\newcommand\dxO{d\bfx_\Omega}
\newcommand\dxOh{d\bfx_\Oh}
\newcommand\n{\mathbf{n}}
\newcommand\dual[2]{\langle#1,#2\rangle}
\newcommand\insieme{\omega}

\newtheorem{remark}{Remark}
\newtheorem{example}{Example}[section]
\newtheorem{assumption}{Assumption}

\newcommand\LG{\begin{color}{blue}}
        \newcommand\GL{\end{color}}

\renewcommand\fnsymbol[1]{\ensuremath{\ifcase\value{#1}\or *\or **\or ***\or ****\else\#\fi}}

\begin{document}

\title{Finite element approximation on piecewise smooth domains}

\author{Daniele Boffi\thanks{Computer, Electrical and Mathematical Sciences
and Engineering, King Abdullah University of Science and Technology (KAUST),
Saudi Arabia; Dipartimento di Matematica ``F.~Casorati'', University of Pavia,
Italy; IMATI-CNR ``Enrico Magenes'', Pavia, Italy
(\email{daniele.boffi@kaust.edu.sa})}
\and
Lucia Gastaldi\thanks{DICATAM, Universit\`a di Brescia, Italy; IMATI-CNR
``Enrico Magenes'', Pavia, Italy
(\email{lucia.gastaldi@unibs.it}}
}

\maketitle

\begin{abstract}
The approximation theory by finite elements is usually described on the
simplified geometry of polygonal or polyhedral domains. While avoiding the
question of boundary approximation makes the analysis simpler, real
applications may involve piecewise smooth domains and in some cases the
extension of classical results to general boundaries may be tricky.

In this paper we discuss in a systematic way the approximation properties and
the error analysis of standard non-isoparametric finite element spaces on
piecewise curvilinear domains in two and three dimensions. We show in detail
how the geometric error introduces a barrier for the $H^1$ and $L^2$ error of
the solution of both Dirichlet and Neumann elliptic problems independently of
the regularity of the problem and of the degree of the finite element spaces
used.
\end{abstract}



\section{Introduction}
\label{se:intro}

The aim of this paper is to discuss in a systematic way the approximation
properties and the error analysis of finite element spaces for the
discretization of partial differential equations defined on piecewise
curvilinear domains in two and three dimensions.

We consider a standard elliptic problem with Dirichlet or Neumann boundary
conditions and perform a finite element analysis on meshes of straight
(non-isoparametric) triangles and tetrahedra.

This is a classical question, already discussed in the pioneering works of
Ciarlet--Raviart~\cite{CiR72} and Strang--Fix~\cite{strangfix}. Other
relevant references include~\cite{Zla73,Zla74,Lenoir,Ber89,BrS94,Cia91}. When
Dirichlet boundary conditions are transferred from the boundary of the
physical domain to the boundary of the computational domain, it is part of the
finite element common knowledge, and it is proved rigorously in some
particular cases, that the $H^1$ norm of the best approximation is capped by
$O(h^{3/2})$, regardless of the regularity of the function to be approximated.

A common cure for this suboptimality is to resort to suitable isoparametric
approximation of the boundary. It is another well-established subject, and
part of the finite element folklore, confirmed under appropriate hypothesis,
that this is sufficient to restore optimal approximation (see, for
instance,~\cite{Scott1975,Lenoir,Cia91,ern-guermond}).
Several other numerical analysis subjects deal with geometrical errors. For
instance, unfitted schemes constitute a growing area where the
transfer of information from different points in the domains is needed,
starting from the Nitsche's method~\cite{nitsche} and the fictitious domain
method~\cite{FD}, and other techniques, such as those
from~\cite{Babuska1973, barellNeunfitted, barellNeuRobin, BrKing94,
burmanhansbo2010, cutfem-acta, SBM, ABG2024, OSZ19}.

In this paper we are interested in the classical finite element method
(straight-sided elements, no isoparametric mapping) and in the sources of
suboptimality for the best approximation and the error of the computed
solution. We remove the convexity and quasi-uniformity assumptions that appear
in classical references and include an $L^2$ duality argument.
We prove rigorously under very general assumptions the already mentioned
$O(h^{3/2})$ barrier for the $H^1$ best approximation under prescribed
boundary conditions and complement it with an $O(h^{5/2})$ barrier for the
$L^2$ best approximation error. Moving to the error of the discretization of a
Dirichlet elliptic problem, we confirm that the $H^1$ barrier is set at
$O(h^{3/2})$, while we show a (less known and perhaps unexpected) stricter
$O(h^2)$ barrier for the $L^2$ error. The same barriers hold for the solution
to the Neumann problem. It is interesting to notice that the $O(h^{3/2})$
barrier for Dirichlet problems comes directly from the same cap for the best
approximation, while the $O(h^{3/2})$ barrier for Neumann problems is not due
to the best approximation, since no boundary conditions are imposed on the
finite element spaces, but derives from a consistency term related to the
geometrical error.

We believe that the novelty of our presentation consists in a clean,
self-contained treatment of the geometric error when a straight-sided mesh
sequence approximates a piecewise smooth (not just convex) domain, starting
from best approximation and Galerkin error estimates, including an $L^2$
duality argument, both for Dirichlet and Neumann boundary conditions.
In particular, we believe that the $L^2$ duality argument, which shares some
of the techniques of~\cite{Lenoir}, presents a result that is not widely known
in our community. Moreover, we make an interesting observation for the second
order approximation of the Neumann problem in two dimensions. Under suitable
hypotheses, valid for instance for a uniform triangulation of a circle, the
optimal $O(h^2)$ rate is obtained for the energy norm error of the solution of
the Neumann problem. This is a very special situation that holds only in two
dimensions and for quadratic polynomials. Namely, the approximation with third
order finite elements is asymptotically worse than with second-order ones.
Further investigation of this phenomenon will be part of a future project.

The structure of the paper is as follows. For the Dirichlet setting,
Section~\ref{se:setting} introduces our problem and sets the notation with a
careful description of the very general assumptions on the domain and the
mesh.
Section~\ref{se:geom} deals with the quantities that will be used for the
estimation of the geometrical error.
The core of our analysis is given by Sections~\ref{se:bestapprox},
\ref{se:main}, and~\ref{se:DL2}, where we discuss the best approximation,
$H^1$ Galerkin, and $L^2$ Galerkin error, respectively, for the Dirichlet
case. The analysis for the Neumann problem, Galerkin approximation both in
$H^1$ and in $L^2$, is presented in Section~\ref{se:neumann}. The extension to
non-homogeneous boundary conditions is sketched in Section~\ref{se:non0}.
Finally, a comprehensive set of numerical tests is reported in
Section~\ref{se:num}.

\section{Dirichlet boundary conditions: setting of the problem}
\label{se:setting}

Most mathematical textbooks on finite element theory start with the
discretization of the Laplace equation on a domain $\Omega$. A common
statement is that for the sake of simplicity $\Omega$ is considered as a
polygonal or polyhedral Lipschitz domain. We aim at investigating in more
detail what happens when the hypotheses on the domain are more general.

We start by discussing a Dirichlet problem for a second order elliptic
equation.
Without loss of generality we consider the Laplace problem: find $u$ such that
\begin{equation}
\aligned
&-\Delta u=f&&\text{in }\Omega\\
&u=g&&\text{on }\pO,
\endaligned
\label{eq:poissong}
\end{equation}
where $\Omega$ is a weakly Lipschitz (bi-Lipschitz image of a half-ball near
each boundary point) domain in $\RE^d$ (d=$2,3$) with a piecewise
regular boundary.

To give a more precise characterization of a piecewise regular boundary, we
introduce suitable assumptions. While in two dimensions the definition is
immediate, its generalization to three dimensions requires dealing with
boundary patches and their boundary edges.

\begin{assumption}
The boundary of $\Omega$ is piecewise $C^2$, that is
$\partial\Omega=\bigcup_{j=1}^{J}\cl{\Gamma_j}$, where the \emph{patches}
$\Gamma_1,\dots,\Gamma_J$ are pairwise disjoint.
For $d=2$ the patches are $C^2$ arcs and the \emph{skeleton}
$\Sigma:=\bigcup_{j}\partial\Gamma_j$ is the set of their endpoints.
When $d=3$, $\Sigma$ consists of finitely many \emph{edges} $E_1,\dots,E_M$:
each $E_m$ is either a closed $C^2$ curve or a $C^2$ arc whose two endpoints
are corners (possibly the same corner), and
$E_m\subset\cl{\Gamma_j}\cap\cl{\Gamma_{j'}}$ for exactly two patches, which
we call \emph{adjacent} to $E_m$.
\label{as:domain}
\end{assumption}

We recall the variational formulation of~\eqref{eq:poissong}. We assume that
$f\in\Ld$ (more general cases could be considered),
$g\in H^{1/2}(\pO)$, and we consider the space $V=\Huo$, the subspace
of $\Hu$ with vanishing trace on $\pO$.
Given a lifting $u_g\in\Hu$ of $g$ to $\Omega$, we look for $u=u_0+u_g$, with
$u_0\in V$ the solution to
\begin{equation}
a(u_0,v)=(f,v)_\Omega-a(u_g,v)\quad\forall v\in V,
\label{eq:poissongvar}
\end{equation}
where
\[
\aligned
&(u,v)_D=\int_D uv\,dx&&\forall u,v\in L^2(D)\\
&a(u,v)=(\Grad u,\Grad v)_\Omega&&\forall u,v\in V.
\endaligned
\]

We consider simplicial triangulations for the numerical approximation
of~\eqref{eq:poissongvar}.
A typical finite element mesh is a triangulation $\Th$ with the property that
each boundary vertex is located on the physical boundary $\pO$ of the
domain $\Omega$. We assume, as is natural, that $\Th$ is consistent with
the parts $\Gamma_j$ that constitute the boundary $\pO$. While the formal
definition is immediate in 2D, the three-dimensional situation should be
carefully described. We state all properties in the following assumption,
where we denote by $\Fbd$ the set of \emph{boundary facets} (edges if $d=2$,
triangles if $d=3$) and, if $d=3$, by $\Ebd$ the set of \emph{boundary edges}
(mesh edges contained in $\partial\Oh$). The set of elements $T$ whose
boundary contains a boundary facet or edge is denoted by $\ThB$ and $\OB$ is
the region occupied by all elements in $\ThB$. Moreover, given a positive
$\delta\le\diam(\Omega_h)$, we denote by $\OBpatch$ the layer containing all
points in $\Oh$ with distance from $\partial\Oh$ less than $\delta$.
\begin{assumption}

The mesh $\Th$ is a family of closed simplices with the following properties.

\begin{itemize}

\item[(M1)]
Every corner of $\partial\Omega$ is a vertex of $\Th$.
An element has at most one boundary facet, and if it has one, its boundary
edges when $d=3$ are the edges of that facet; an element without a boundary
facet has at most one boundary edge when $d=3$ or one boundary vertex when
$d=2$.

\item[(M2)]
Each facet $F\in\Fbd$ has all vertices on a single $\Gamma_j$. We denote by
$\Gamma_{j,h}$, for $j=1,\dots,J$, the union of all the facets $F$ with all 
the vertices on $\Gamma_j$, so that $\bigcup_{j=1}^J\Gamma_{j,h}=\pOh$.

\item[(M3)]
For $d=2$, the discrete patch $\Gamma_{j,h}$ is the
polygonal curve through the vertices of $\Th$ lying on $\cl{\Gamma_j}$, in
their order along $\Gamma_j$, and the discrete skeleton $\Sigma_h$ is the set of
vertices shared by two boundary edges belonging to different $\Gamma_{j,h}$.

\item[(M4)]
For $d=3$, a boundary edge shared by two boundary facets belonging to two
different $\Gamma_j$ is called a \emph{skeleton edge}; $\Sigma_h$ denotes the
union of the skeleton edges.
For every edge $E_m$, the union $E_{m,h}$ of the skeleton edges
whose endpoints lie on $E_m$ is a polygonal curve whose vertices are exactly
the vertices of $\Th$ lying on $E_m$, joined in their order along $E_m$
(closed if $E_m$ is closed, with the same endpoints as $E_m$ otherwise); every
skeleton edge belongs to exactly one $E_{m,h}$.
\end{itemize}
\label{as:lunette}
\end{assumption}
Unless isoparametric approximations are taken into account, the computational
domain $\Oh$ (the union of all elements of the mesh $\Th$) is a polygon or a
polyhedron with straight sides/faces.
Clearly, if $\Omega$ is convex, then $\Oh$ is inscribed in $\Omega$,
while in general there might be parts of $\Oh$ that are outside $\Omega$.

A finite element discretization of~\eqref{eq:poissongvar} involves the
following ingredients: a finite element space $W_h$, approximations $f_h$ and
$g_h$ of our data, and a discrete lifting $u_{g,h}\in W_h$ of $g_h$ to
$\Oh$. We denote by $k$ the maximum degree of the piecewise polynomials in
$W_h$. After considering the subspace $V_h$ of $W_h$ consisting of the
functions vanishing on $\pOh$, that is,
\[
V_h=\{v\in W_h:v|_{\pOh}=0\},
\]
the discrete solution $u_h$ is obtained as the sum of $u_{g,h}$ and the
solution $u_{0,h}\in V_h$ to
\begin{equation}
a_h(u_{0,h},v)=(f_h,v)_{\Oh}-a_h(u_{g,h},v)\quad\forall v\in V_h,
\label{eq:poissongvarh}
\end{equation}
where
\[
a_h(u,v)=(\Grad u,\Grad v)_{\Oh}\quad\forall u,v\in V_h.
\]

One of the first questions for the implementation of any finite element code
is the definition of $f_h$ and $g_h$ when $\Omega$ is not already a
polygon/polyhedron and is replaced by $\Oh$. At this level we are not
interested in numerical integration, but only in the error induced by the
geometry approximation.

If $\Omega$ is convex, then $\Oh$ is contained in $\Omega$ and it is
immediate to define $f_h$ as the restriction of $f$ to $\Oh$. Otherwise,
the definition of $f_h$ involves an extension of $f$ to the union
$\Obar$. In the presence of lower order methods an extension by zero is
viable, while if higher accuracy is desired, then we expect $f$ to be more
regular than simply being in $\Ld$ and we can consider more sophisticated
extensions if the domain is smooth enough (see, for
instance,~\cite{Calderon,Stein,Jones,Rogers}).
Since our main focus in this paper is on boundary conditions, from now on
we consider $f$ already defined on $\Obar$ and, consequently, we take $f_h=f$.

When $g=0$ the natural choice is to consider also $g_h=0$. Otherwise, there
are at least two options. If $g$ is the restriction to $\pO$ of a
function $\overline{g}$ defined on $\Obar$, then $g_h$ can be taken as the
restriction of $\overline{g}$ to $\pOh$. From a practical point of view, the
situation is different when lowest- or higher-order methods are used. For
instance, in the case of continuous piecewise linear finite elements, the
nodes are placed at the vertices of the elements, and the nodes on
$\partial\Omega_h$ are also on $\partial\Omega$. Hence it is possible to
define $g_h$ on $\partial\Omega_h$ by just interpolating $g$.
On the other hand, in the case of higher-order elements, some internal nodes
of the boundary facets are on $\partial\Omega_h$ but do not necessarily lie on
$\partial\Omega$. This implies that the value of $g_h$ on those nodes should
be somehow deduced from the exact value of $g$ on $\partial\Omega$.
In all these cases we have that $u-u_h$ vanishes neither on $\partial\Omega$
nor on $\partial\Omega_h$ and the analysis should take into account this
inconsistency.

A first step towards the analysis of the error between the continuous solution
$u\in\Hu$ and the discrete one $u_h\in W_h$ involves the study of the best
approximation of $u$. In Section~\ref{se:bestapprox} we perform this task and
continue later with the convergence analysis.

\section{Geometrical properties related to $\Omega$ and $\Oh$}
\label{se:geom}

In this section we collect results related to the mismatch between $\Omega$
and $\Oh$. We discuss in particular the behavior of boundary elements; some
possible configurations are shown in Figures~\ref{fg:onecase2D}
and~\ref{fg:onecase3D}.

\begin{figure}

\begin{center}
{\tikzset{every picture/.append style={scale=0.65}}
\providecolor{bnd}{RGB}{170,40,40}%
\providecolor{lift}{RGB}{31,90,160}%
\providecolor{lightlift}{RGB}{200,218,240}%
\begin{tikzpicture}[
    x=1cm, y=1cm,
    boundary/.style={draw=bnd, line width=1.1pt, line cap=round},
    element/.style={draw=black, line width=0.7pt, line join=round},
    lbl/.style={font=\footnotesize, inner sep=1.5pt},
    leader/.style={draw=black!60, line width=0.3pt},
  ]
  \pgfmathsetmacro\tdL{4.6}      
  \pgfmathsetmacro\tdC{0.45}     
  \pgfmathsetmacro\tdK{1.45}     
  \pgfmathsetmacro\tdE{0.55}     
  \pgfmathsetmacro\tdEnd{\tdL+\tdE}
  \tikzset{declare function={bd(\x)=\tdK*\tdL*(\x/\tdL)*(1-\x/\tdL)*(\x/\tdL-\tdC);}}
  \coordinate (a1) at (0,0);
  \coordinate (a2) at (\tdL,0);
  \coordinate (a3) at (1.95,-3.1);
  \fill[lightlift] (a3) -- plot[domain=0:\tdL, samples=120] (\x,{bd(\x)}) -- cycle;
  \draw[boundary] plot[domain=-\tdE:\tdEnd, samples=160] (\x,{bd(\x)});
  \draw[element] (a1) -- (a2) -- (a3) -- cycle;
  \foreach \v in {a1,a2,a3} \fill (\v) circle (1.4pt);
  \path (a1) node[lbl, below left=1pt] {$a_1$}
        (a2) node[lbl, above right=1pt] {$a_2$}
        (a3) node[lbl, below=1pt] {$a_3$}
        (barycentric cs:a1=1,a3=1) node[lbl, left=3pt] {$T$}
        (barycentric cs:a2=1,a3=1) node[lbl, right=3pt, text=lift] {$\tilde T$}
        (-\tdE,{bd(-\tdE)}) node[lbl, left=1pt, text=bnd] {$\partial\Omega$};
  \pgfmathsetmacro\tdXp{0.775*\tdL}
  \pgfmathsetmacro\tdXm{0.20*\tdL}
  \node[lbl, anchor=south] (lp) at (\tdXp+0.15,0.95) {$\tilde T\setminus T$};
  \draw[leader] (lp.south) -- (\tdXp,{bd(\tdXp)/2});
  \node[lbl, anchor=south] (lm) at (\tdXm+0.35,0.95) {$T\setminus\tilde T$};
  \draw[leader] (lm.south) -- (\tdXm,{bd(\tdXm)/2});
\end{tikzpicture}%
}
\end{center}

\caption{A boundary element $T$ and $\Tt$ in two dimensions.}
\label{fg:onecase2D}
\end{figure}
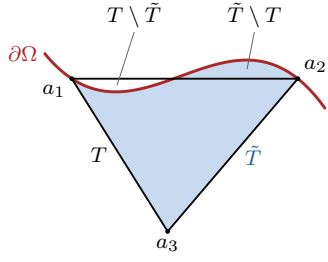

\begin{figure}

\begin{center}
{\tikzset{every picture/.append style={scale=0.65}}
\input{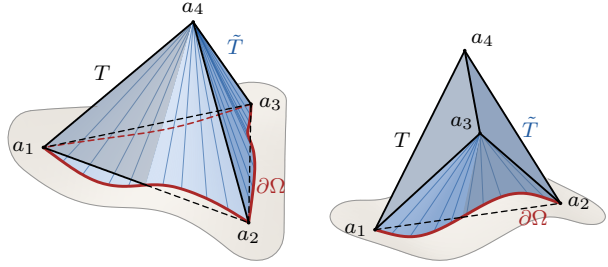}}
\end{center}

\caption{An element $T$ with a face (left) or an edge (right) on $\partial\Oh$
and the corresponding $\Tt$ in three dimensions. The shaded part of the front
face of $\Tt$ is behind the corresponding face of $T$, while the other part is
in front.}
\label{fg:onecase3D}
\end{figure}

We consider a mapping $\Phi_h:\Oh\to\Omega$ acting on boundary elements
$T\in\ThB$ as shown in Figure~\ref{fg:onecase2D} in two dimensions
and Figure~\ref{fg:onecase3D} in three dimensions, where $\Tt$ is the
image of $T$:
\begin{equation}
\Tt=\Phi_h(T).
\label{eq:Tt}
\end{equation}
The construction of $\Phi_h$ can be performed by putting together tools
already introduced in the literature. On elements in $\OB$ not touching a
skeleton edge the construction is given by Lenoir~\cite{Lenoir} for smooth
boundaries (see also Scott~\cite{Scott1975} in two dimensions). Related works
are~\cite{Dubois1990,ElliottRanner2013,CaubetGhantousPierre2023,
GordonHall1973,LiQiuXieYu2024,BonitoCasconMekchayMorinNochetto2016}.

\begin{remark}
While in two dimensions $\Phi_h$ preserves the edges of $T\in\ThB$ that are
not along the boundary, in three dimensions the internal faces of a boundary
tetrahedron $T\in\ThB$ are mapped to possibly non-planar \emph{cones}. It
follows that even for a convex domain is not true that $\Phi_h(T)$ contains
$T$; hence the mismatch between $\Omega$ and $\Oh$ cannot be described in
terms of $\Tt\setminus T$ and $T\setminus\Tt$.
A better description introduces, for each $T\in\ThB$, the part of $\Tt$
outside $\Oh$, namely $\Spt=\Tt\setminus\cl\Oh$, and the corresponding part of
$T$ outside $\Omega$, that is $\Smt=T\setminus\cl\Omega$. In two dimensions
$\Spt=\Tt\setminus T$ and $\Smt=T\setminus\Tt$.
\label{re:mismatch}
\end{remark}

We denote the collection of all regions of the form $\Spt$ and $\Smt$ by
$\lunette$. This set is naturally formed by the following two subsets:
\[
\aligned
&\lunetteconv=\{\omega\in\lunette:\omega\subset\OmOh\}\\
&\lunetteconc=\{\omega\in\lunette:\omega\subset\OhmO\},
\endaligned
\]
so that $\lunetteconv$ is the set of the $\Spt$'s and $\lunetteconc$ of the
$\Smt$'s.

The usual notation $|\omega|$ will be adopted for the $d$-measure of a set
$\omega\subset\RE^d$.
By abuse of notation, we indicate by $|\lunette|$ the measure of the region
occupied by the elements in $\lunette$. This region is equal to
$(\OmOh)\cup(\OhmO)$. Analogously, we will use
$|\lunetteconv|$ and $|\lunetteconc|$ so that we have
\[
|\lunette|=|\lunetteconv|+|\lunetteconc|.
\]

Estimates involving the measure of $\lunette$ require a bound on the
geometrical error. We state this classical result in the following lemma,
since it will be used frequently in the rest of the paper.

\begin{lemma}

If $\partial\Omega$ is piecewise $C^2$ and the mesh is compatible with the
domain, see Assumption~\ref{as:lunette}, then for any $\omega\in\lunette$ the
distance of $\partial\Oh\cap\partial\omega$ from
$\partial\Omega\cap\partial\omega$ is of order $h_\omega^2$.
Consequently, $|\omega|\le Ch_\omega^{d+1}$ and $|\lunette|$ is of order
$h^2$.

\label{le:lunette}
\end{lemma}

\begin{proof}

This is a classical result which we report here to help the reader since
it is often quoted without proof.

We start from $d=2$, so that the boundary elements are as in
Figure~\ref{fg:onecase2D}.
Let us consider the $C^2$ function $\phi:[0,\ell]\to\RE$ representing
$\partial\Omega$ between $a_1$ and $a_2$ so that $\phi(0)=0$ corresponds to
$a_1$ and $\phi(\ell)=0$ corresponds to $a_2$. Let $\ell^*\in(0,\ell)$ be a
maximum point for $|\phi|$ so that $\phi'(\ell^*)=0$. It turns out that
\[
0=\phi(0)=\phi(\ell^*)-\ell^*\phi'(\ell^*)+(\ell^*)^2\phi''(\xi)/2
\]
for some $\xi\in(0,\ell)$. Since $\ell^*<\ell\le h_\omega$ and $\phi''$ is
bounded we get
\begin{equation}
|\phi(\ell^*)|\le Ch_\omega^2.
\label{eq:hquadro}
\end{equation}

This proof easily extends to $d=3$ where the function $\phi$ is defined on a
two-dimensional region $T$ of diameter $h_\omega$. After identifying a
maximum point $(\ell^*,\mathscr{m}^*)$ of $|\phi|$, the same argument as before can be
applied to the restriction of $\phi$ to a segment connecting a vertex of $T$
to $(\ell^*,\mathscr{m}^*)$.

The set $\omega$ is a region contained in a box with base of size
$O(h_\omega^{d-1})$ and height $O(h_\omega^2)$, hence $|\omega|\le
Ch_\omega^{d+1}$.

The result concerning $|\lunette|$, that is, the measure of
$(\OmOh)\cup(\OhmO)$, is obtained by observing that
the union of the regions bounded by $\partial\Omega$ and $\partial\Oh$
is always contained in a boundary layer of thickness $Ch^2$ centered at
$\partial\Omega$, so that
\[
|\lunette|\le C|\partial\Omega|h^2.
\]
\end{proof}

We state in the next lemma and proposition (see~\cite{Lenoir}) some
consequences related to the behavior of the map $\Phi_h$.

\begin{lemma}\label{le:local}
There is $h_0>0$ such that for $h\le h_0$ and $T\in\ThB$:
\begin{itemize}
\item[(a)] $|S_T^\pm|\le Ch_T^{d+1}$ and $|\Phi_h^{-1}(\Spt)|\le Ch_T^{d+1}$;
\item[(b)] $\dist(y,\partial\Oh)\le Ch_T^2$ for every
  $y\in\Smt\cup\Phi_h^{-1}(\Spt)$;
\item[(c)] for $v\in V_h$, $y\in T$, and $z\in\R^d$ with $|y-z|\le Ch_T^2$,
\[
|v(y)-v(z)|\le C|y-z||T|^{-1/2}\|\Grad v\|_{L^2(\patch T)};
\]
in particular
$\|v\|_{L^\infty(S)}\le Ch_T^2\,|T|^{-1/2}\|\Grad v\|_{L^2(\patch T)}$
for $S=\Smt$ and for $S=\Phi_h^{-1}(\Spt)$,
\end{itemize}
where $\patch T$ denotes the union of elements whose closure intersects
$\cl{T}$.
\end{lemma}

\begin{proposition}

Let Assumptions~\ref{as:domain} and \ref{as:lunette} hold. There are $h_0>0$
and $C>0$, depending only on $\Omega$ and the shape-regularity of the mesh,
such that for all $h\le h_0$:
\begin{itemize}
\item[(a)] $\Phi_h$ is a bijection of $\cl{\Oh}$ onto $\cl{\Omega}$, Lipschitz
continuous with a Lipschitz continuous inverse and
$\Phi_h(\cl{\Gamma_{j,h}})=\cl{\Gamma_j}$ for $j=1,\dots,J$;
\item[(b)] $\Phi_h=\Id$, the identity operator, on $\cl{\Oh}\setminus\OB$ and
$|\OB|\le Ch$;
\item[(c)] we have
\[
\aligned
&\|\Phi_h-\Id\|_{L^\infty(\Oh)}\le Ch^2 \\
&\|\Jac\Phi_h-\matId\|_{L^\infty(\Oh)}\le Ch\\
&\bigl\||\det\Jac\Phi_h|-1\bigr\|_{L^\infty(\Oh)}\le Ch
\endaligned
\]
so that
\[
\left\|\frac{\Jac\Phi_h\Jac\Phi_h^{\top}}{|\det\Jac\Phi_h|}-\matId\right\|_{L^\infty(\Oh)}\le
Ch,
\]
where $\Jac$ and $\matId$ are the Jacobian and identity matrices,
respectively.
\end{itemize}
\label{pr:phih}
\end{proposition}

The following bounds will be used: $|\det D\Phi_h|$, $\|D\Phi_h\|$, and
$\|D\Phi_h^{-1}\|$ are bounded above and below by constants independent of
$h$, so that for every measurable $\insieme\subset\Oh$ and every $v\in H^1(\Oh)$
\begin{equation}\label{eq:cov}
\aligned
&|\Phi_h(\insieme)|\le C|\insieme|\\
&\|v\circ\Phi_h^{-1}\|_{L^2(\Phi_h(\insieme))}\le C\|v\|_{L^2(\insieme)}\\
&\|\Grad(v\circ\Phi_h^{-1})\|_{L^2(\Phi_h(\insieme))}\le C\|\Grad v\|_{L^2(\insieme)}
\endaligned
\end{equation}
and symmetrically for $g\mapsto g\circ\Phi_h$.

For a given $g:\Omega\to\RE$, we consider $\lenoir{g}=g\circ\Phi_h:\Oh\to\RE$,
so that
\begin{equation}
\|\lenoir{g}\|_{L^2(T)}\le C\|g\|_{L^2(\Tt)}.
\label{eq:lenoir}
\end{equation}

\section{Best approximation estimates with boundary conditions}
\label{se:bestapprox}

A fundamental question is how finite element functions $v_h$ defined on a
triangulation that does not fit the domain $\Omega$ can approximate a
smooth function $v$ defined on $\Omega$. Unless $\Omega$ is convex, it is not
guaranteed that $\Oh$ is contained in $\Omega$. It is then natural to consider
approximation estimates on the largest subset of $\Omega$ where both $v$ and
$v_h$ are defined, namely $\OOh$.

We revisit and generalize to any degree (for $d=2,3$) a result sketched
in~\cite{strangfix} about the approximation properties of piecewise quadratic
finite elements in two dimensions. While~\cite{strangfix} assumes a convex
domain $\Omega$ and a quasi-uniform mesh, we do not need such conditions.

\subsection{Interpolation estimates with homogeneous boundary conditions on
general domains}

Let $\Omega\subset\RE^d$ ($d=2,3$) be a domain satisfying
Assumption~\ref{as:domain} and $\Omega_h$ the triangulated domain.

Given a smooth function $v\in H^{k+1}(\Omega)\cap\Huo$, we are interested in
defining its interpolant of order $k\ge2$ in $\Oh$ and in estimating the
interpolation error in $\OOh$. As explained above, the natural choice when
homogeneous boundary conditions are considered on $\partial\Omega$ ($g=0$) is
to impose on discrete functions homogeneous boundary conditions on $\pOh$,
that is, $g_h=0$.

The first issue that we face when trying to define an interpolation operator
is that some interpolation nodes may lie outside $\Omega$. For this reason, we
introduce an extension operator.

\begin{assumption}

There exists an extension operator
$\extension:W^{s,p}(\Omega)\to W^{s,p}(\RE^d)$ such that
$\extension v=v$ in $\Omega$ and
\begin{equation}
\|\extension v\|_{W^{s,p}(\RE^d)}\le C\|v\|_{W^{s,p}(\Omega)}
\label{eq:Stein}
\end{equation}
for $p\in[1,+\infty]$ and all nonnegative integers $s$, with a constant $C$
depending on $d$, $s$, and $\Omega$.

\label{as:ext}
\end{assumption}

\begin{remark}
If $\Omega$ is a bounded domain of class $C^1$ then the existence of the
extension in Assumption~\ref{as:ext} is a consequence of the more general
results by Stein in~\cite[Theorem~5, p.~181]{Stein}.
In our setting of a weakly Lipschitz domain such result was proved by
Jones~\cite{Jones} and later by Rogers~\cite{Rogers} with stability
constant independent of $s$. Essentially, these results cover most of the
cases of interest in applications with the exception of situations like
cusps in 2D or their analogues in 3D.
\end{remark}

We denote by $\vi$ the usual piecewise polynomial interpolant in $\Omega_h$
which is well defined for a convex $\Omega$ and $v\in H^{k+1}(\Omega)$. For a
more general $\Omega$, we can define $\vi$ in $\Omega_h$ as the interpolant of
$\extension v\in H^{k+1}(\RE^d)$ in order to allow for non-convex domains.
Standard interpolation estimates give
\begin{equation}
\aligned
&\|\extension v-\vi\|_{L^2(T)}\le Ch^{k+1}|\extension v|_{H^{k+1}(T)}\\
&|\extension v-\vi|_{H^1(T)}\le Ch^k|\extension v|_{H^{k+1}(T)}
\endaligned
\label{eq:extinterp}
\end{equation}
for all $T\in\Th$.

It is convenient to introduce also the interpolant $\vit\in V_h$ of $v$,
defined as follows:
\begin{enumerate}
\item[i)] $\vit=0$ at all nodes lying on $\pOh$ or outside $\Omega$;
\item[ii)] $\vit=v$ at all nodes of $\Oh$ not on $\pOh$ and internal to
$\Omega$.
\end{enumerate}

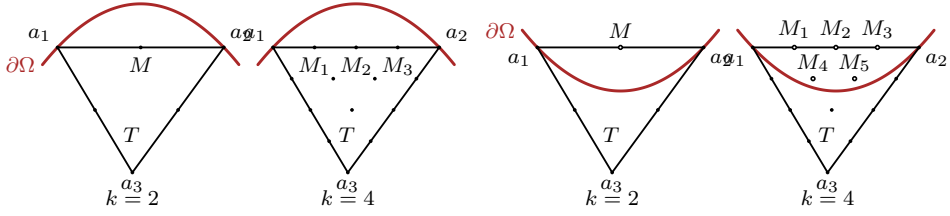
\begin{figure}

\begin{center}
{\tikzset{every picture/.append style={scale=0.5}}
\providecolor{bnd}{RGB}{170,40,40}%
\ifcsname c@dcnode\endcsname\else\newcounter{dcnode}\fi
\begin{tikzpicture}[
    x=1cm, y=1cm,
    boundary/.style={draw=bnd, line width=1.1pt, line cap=round},
    element/.style={draw=black, line width=0.7pt, line join=round},
    lbl/.style={font=\footnotesize, inner sep=1.5pt},
    dof out/.style={fill=white, draw=black, line width=0.6pt},
  ]
  \pgfmathtruncatemacro\dcSide{-1}
  \pgfmathsetmacro\dcL{4.4}     
  \pgfmathsetmacro\dcH{3.3}     
  \pgfmathsetmacro\dcX{1.98}    
  \pgfmathsetmacro\dcS{1.15}    
  \pgfmathsetmacro\dcE{0.4}     
  \pgfmathsetmacro\dcGap{1.3}   
  \pgfmathsetmacro\dcEnd{\dcL+\dcE}
  \tikzset{declare function={bd(\x)=-\dcSide*4*\dcS*\x*(\dcL-\x)/(\dcL*\dcL);}}
  \ifnum\dcSide=1                
    \tikzset{e label/.style={above=2pt}, a1 label/.style={below left=1pt},
             a2 label/.style={below right=1pt}}
  \else
    \tikzset{e label/.style={below=2pt}, a1 label/.style={above left=1pt},
             a2 label/.style={above right=1pt}}
  \fi
  \foreach \k [count=\n from 0] in {2,4} {
    \pgfmathsetmacro\dcShift{\n*(\dcL+\dcGap)}
    \begin{scope}[shift={(\dcShift,0)}]
      \draw[boundary] plot[domain=-\dcE:\dcEnd, samples=120] (\x,{bd(\x)});
      \draw[element] (0,0) -- (\dcL,0) -- (\dcX,-\dcH) -- cycle;
      \setcounter{dcnode}{0}%
      \foreach \l in {0,...,\k} {
        \pgfmathtruncatemacro\jmax{\k-\l}
        \foreach \j in {0,...,\jmax} {
          \pgfmathsetmacro\px{(\j*\dcL+\l*\dcX)/\k}
          \pgfmathsetmacro\py{-\l*\dcH/\k}
          \pgfmathtruncatemacro\out{\py > bd(\px)+0.001}
          \pgfmathtruncatemacro\onE{(\l==0) && (\j>0) && (\j<\k)}
          \ifnum\out=1
            \path[dof out] (\px,\py) circle (1.4pt);
          \else
            \fill (\px,\py) circle (1.4pt);
          \fi
          \pgfmathtruncatemacro\lab{\out || \onE}
          \ifnum\lab=1
            \stepcounter{dcnode}%
            \ifnum\k=2 \def\dcName{M}\else \edef\dcName{M_{\arabic{dcnode}}}\fi
            \ifnum\onE=1
              \path (\px,\py) node[lbl, e label] {$\dcName$};
            \else            
              \path (\px,\py) node[lbl, above=1pt] {$\dcName$};
            \fi
          \fi
        }
      }
      \path (0,0) node[lbl, a1 label] {$a_1$}
            (\dcL,0) node[lbl, a2 label] {$a_2$}
            (\dcX,-\dcH) node[lbl, below=1pt] {$a_3$}
            ({0.13*\dcL+0.7*\dcX},{-0.7*\dcH}) node[lbl] {$T$}
            (\dcX,{-\dcH-0.65}) node[lbl] {$k=\k$};
      \ifnum\n=0
        \path (-\dcE,{bd(-\dcE)}) node[lbl, left=1pt, text=bnd] {$\partial\Omega$};
      \fi
    \end{scope}
  }
\end{tikzpicture}%
}%
%
%
{\tikzset{every picture/.append style={scale=0.5}}
\providecolor{bnd}{RGB}{170,40,40}%
\ifcsname c@dcnode\endcsname\else\newcounter{dcnode}\fi
\begin{tikzpicture}[
    x=1cm, y=1cm,
    boundary/.style={draw=bnd, line width=1.1pt, line cap=round},
    element/.style={draw=black, line width=0.7pt, line join=round},
    lbl/.style={font=\footnotesize, inner sep=1.5pt},
    dof out/.style={fill=white, draw=black, line width=0.6pt},
  ]
  \pgfmathtruncatemacro\dcSide{1}
  \pgfmathsetmacro\dcL{4.4}     
  \pgfmathsetmacro\dcH{3.3}     
  \pgfmathsetmacro\dcX{1.98}    
  \pgfmathsetmacro\dcS{1.15}    
  \pgfmathsetmacro\dcE{0.4}     
  \pgfmathsetmacro\dcGap{1.3}   
  \pgfmathsetmacro\dcEnd{\dcL+\dcE}
  \tikzset{declare function={bd(\x)=-\dcSide*4*\dcS*\x*(\dcL-\x)/(\dcL*\dcL);}}
  \ifnum\dcSide=1                
    \tikzset{e label/.style={above=2pt}, a1 label/.style={below left=1pt},
             a2 label/.style={below right=1pt}}
  \else
    \tikzset{e label/.style={below=2pt}, a1 label/.style={above left=1pt},
             a2 label/.style={above right=1pt}}
  \fi
  \foreach \k [count=\n from 0] in {2,4} {
    \pgfmathsetmacro\dcShift{\n*(\dcL+\dcGap)}
    \begin{scope}[shift={(\dcShift,0)}]
      \draw[boundary] plot[domain=-\dcE:\dcEnd, samples=120] (\x,{bd(\x)});
      \draw[element] (0,0) -- (\dcL,0) -- (\dcX,-\dcH) -- cycle;
      \setcounter{dcnode}{0}%
      \foreach \l in {0,...,\k} {
        \pgfmathtruncatemacro\jmax{\k-\l}
        \foreach \j in {0,...,\jmax} {
          \pgfmathsetmacro\px{(\j*\dcL+\l*\dcX)/\k}
          \pgfmathsetmacro\py{-\l*\dcH/\k}
          \pgfmathtruncatemacro\out{\py > bd(\px)+0.001}
          \pgfmathtruncatemacro\onE{(\l==0) && (\j>0) && (\j<\k)}
          \ifnum\out=1
            \path[dof out] (\px,\py) circle (1.4pt);
          \else
            \fill (\px,\py) circle (1.4pt);
          \fi
          \pgfmathtruncatemacro\lab{\out || \onE}
          \ifnum\lab=1
            \stepcounter{dcnode}%
            \ifnum\k=2 \def\dcName{M}\else \edef\dcName{M_{\arabic{dcnode}}}\fi
            \ifnum\onE=1
              \path (\px,\py) node[lbl, e label] {$\dcName$};
            \else            
              \path (\px,\py) node[lbl, above=1pt] {$\dcName$};
            \fi
          \fi
        }
      }
      \path (0,0) node[lbl, a1 label] {$a_1$}
            (\dcL,0) node[lbl, a2 label] {$a_2$}
            (\dcX,-\dcH) node[lbl, below=1pt] {$a_3$}
            ({0.13*\dcL+0.7*\dcX},{-0.7*\dcH}) node[lbl] {$T$}
            (\dcX,{-\dcH-0.65}) node[lbl] {$k=\k$};
      \ifnum\n=0
        \path (-\dcE,{bd(-\dcE)}) node[lbl, left=1pt, text=bnd] {$\partial\Omega$};
      \fi
    \end{scope}
  }
\end{tikzpicture}%
}
\end{center}

\caption{Boundary elements in two dimensions and the corresponding degrees of
freedom.}
\label{fg:twocasesbis}
\end{figure}

In~\cite{strangfix} the case $d=2$ and $k=2$ is discussed when $\Omega$ is
convex. In our notation, the interpolants $\vi$ and $\vit$ in
that case differ only at the node $M$ in each triangle $T\in\ThB$ which is
always of the form of the left triangle of Figure~\ref{fg:twocasesbis}. For
$\Omega$ convex, the following suboptimal estimate is stated
in~\cite{strangfix}:
\begin{equation}
\|v-\vit\|_{H^1(\OOh)}\le Ch^{3/2}\|v\|_{H^3(\Omega)}.
\end{equation}

The next theorem states this estimate in general together with the
corresponding $L^2$ bound.

\begin{theorem}
Let $v$ be a function in $H^3(\Omega)\cap\Huo$ and $\vit$ the interpolant
introduced above. Then
\begin{equation}
\aligned
&\|v-\vit\|_{L^2(\OOh)}\le Ch^{5/2}\|v\|_{H^3(\Omega)}\\
&|v-\vit|_{H^1(\OOh)}\le Ch^{3/2}\|v\|_{H^3(\Omega)}.
\endaligned
\label{eq:prima}
\end{equation}
for any $k\ge2$ and $d=2,3$.

Moreover, considering the extension operator $\extension$ introduced in
Assumption~\ref{as:ext}, we have
\begin{equation}
\aligned
&\|\extension v-\vit\|_{L^2(\Oh)}\le Ch^{5/2}\|v\|_{H^3(\Omega)}\\
&|\extension v-\vit|_{H^1(\Oh)}\le Ch^{3/2}\|v\|_{H^3(\Omega)}.
\endaligned
\label{eq:pariopp}
\end{equation}
\label{th:bestapprox}
\end{theorem}

\begin{proof}

In Lemma~\ref{le:lunette} we proved that the distance of $\partial\Omega$ from
$\partial\Oh$ decays as $O(h^2)$.
We estimate as usual $\|v-\vit\|$ element by element and remark that the only
difference with respect to the usual interpolation estimate occurs in $\ThB$.
Let us consider a boundary element $T\in\ThB$ and estimate the difference
$\|\vi-\vit\|$ in $T$.
We denote by $\varphi_i$, $i=1,\dots,N$, the $N$ shape functions associated
with the nodes of $T$.
For ease of notation we omit any subscript referring to
$T$. Without loss of generality, we assume that the first $N_T$ shape
functions correspond to all the nodes $M_1,\dots,M_{N_T}$ where $\vi$ and
$\vit$ are different.
For instance, in the two cases of Figure~\ref{fg:twocasesbis} on the left,
$N_T=1$ and the first node $M_1$ is $M$. More general cases are reported on
the right in Figure~\ref{fg:twocasesbis}.
In the two-dimensional case, all $\varphi_i$'s ($i=1,\dots,N_T$) are
supported in $T$ --- indeed, all involved nodes are either internal or on the
boundary of $\Oh$ and not vertices of $T$. In three dimensions we cut off the
$\varphi_i$'s to zero outside $T$ so that we can use the same formalism.

From the regularity of $v$ and the distance from $M_i$ ($i=1,\dots,N_T$) to
the boundary $\partial\Omega$ estimated in Lemma~\ref{le:lunette}, it follows
that
\[
|(\vi-\vit)(M_i)|\le Ch^2\|\Grad\extension v\|_{L^\infty(\Obar)}.
\]
We have
\[
\aligned
\|\vi-\vit\|_{L^2(\Oh)}&=\Big\|\sum_{T\in\ThB}\sum_{i=1}^{N_T}v(M_i)\varphi_i\Big\|_{L^2(\Oh)}\\
&\le
Ch^2\|\Grad\extension
v\|_{L^\infty(\Obar)}\Big\|\sum_{T\in\ThB}\sum_{i=1}^{N_T}|\varphi_i|\Big\|_{L^2(\Oh)}.
\endaligned
\]
Since $|\varphi_i|\le1$ and the support of
$\sum_{T\in\ThB}\sum_{i=1}^{N_T}|\varphi_i|$ is contained in a boundary layer
of measure bounded by $Ch$, we obtain
\begin{equation}
\Big\|\sum_{T\in\ThB}\sum_{i=1}^{N_T}|\varphi_i|\Big\|_{L^2(\Oh)}\le Ch^{1/2}
\label{eq:phiscaling}
\end{equation}
so that we get
\[
\|\vi-\vit\|_{L^2(\Oh)}\le Ch^{5/2}\|\Grad\extension v\|_{L^\infty(\Obar)}.
\]

Similarly, we can obtain the estimate for $|\vi-\vit|_{1,\OOh}$. The only
difference is the analogue of~\eqref{eq:phiscaling} which now reads
\[
\Big\vert\sum_{T\in\ThB}\sum_{i=1}^{N_T}|\varphi_i|\Big\vert_{H^1(\Oh)}\le
Ch^{-1/2}.
\]

Taking into account the continuous embedding of $H^2(\Obar)$ into
$L^\infty(\Obar)$, so that $\|\Grad\extension v\|_{L^\infty(\Obar)}$ is
bounded by $\|\extension v\|_{H^3(\Obar)}$, we obtain
\[
\aligned
&\|\vi-\vit\|_{L^2(\Oh)}\le Ch^{5/2}\|\extension v\|_{H^3(\Obar)}\\
&|\vi-\vit|_{H^1(\Oh)}\le Ch^{3/2}\|\extension v\|_{H^3(\Obar)}.
\endaligned
\]

The final estimate follows by considering the boundedness of the
extension (see Assumption~\ref{as:ext}) and that $\OOh$ is included in
$\Oh$:
\[
\|v-\vit\|_{L^2(\OOh)}\le Ch^{5/2}\|v\|_{H^3(\Omega)}.
\]
Analogously, we get the $H^1$ estimate
\[
|v-\vit|_{H^1(\OOh)}\le Ch^{3/2}\|v\|_{H^3(\Omega)}.
\]

The bound~\eqref{eq:pariopp} follows by the triangle inequality, taking into
account~\eqref{eq:extinterp}. The estimate~\eqref{eq:prima} is just the
restriction of~\eqref{eq:pariopp} to $\OOh$.
\end{proof}

\begin{remark}
The above estimate gives a barrier on the best approximation for higher-order
polynomial approximation, independently of the smoothness of $v$. The cap
comes from the distance between the boundaries of $\Omega$ and $\Oh$.
This barrier is the standard motivation for isoparametric elements, which we
do not pursue here.
\label{re:barrier}
\end{remark}

\section{Error analysis: $H^1$ estimate with Dirichlet boundary conditions}
\label{se:main}

This section studies the error between the exact solution
of~\eqref{eq:poissong} and the numerical solution of~\eqref{eq:poissongvarh}
with homogeneous boundary condition $g=0$.
The first question is how to compute the error between $u\in\Huo$ and $u_h\in
V_h$, since $u$ is defined in $\Omega$ and $u_h$ in $\Oh$. A natural
choice is to use $\extension{u}$ since $\Oh$ may not be contained in $\Omega$.

The following theorem gives a bound for the energy norm of the error.

\begin{theorem}\label{th:main}
Let $u\in H^1_0(\Omega)\cap H^2(\Omega)$ be the solution of~\eqref{eq:poissongvar}
with $g=0$ and $u_h\in V_h$ the solution of~\eqref{eq:poissongvarh}. Then
\begin{equation}\label{eq:strangD}
\|\extension u-u_h\|_{H^1(\Omega_h)}\le
C\inf_{v_h\in V_h}\|\extension u-v_h\|_{H^1(\Omega_h)}
+C\,h^{5/2}\,\|f+\Delta\extension u\|_{L^2(\Omega_h\setminus\Omega)}.
\end{equation}
In particular, if $f=-\Delta\extension u$ in $\Omega_h\setminus\Omega$, the error is bounded
by the best approximation alone.
\end{theorem}

\begin{proof}
We start with a Strang-like estimate. For $v_h\in V_h$ let $w_h:=v_h-u_h$. The
coercivity of $a_h$ on $H^1_0(\Omega_h)$ and~\eqref{eq:poissongvarh} give
\[
\alpha\|w_h\|^2_{H^1(\Omega_h)}\le a_h(w_h,w_h)=a_h(v_h-\extension u,w_h)+\delta_h(w_h)
\]
with
\[
\delta_h(w_h):=a_h(\extension u,w_h)-(f,w_h)_{\Omega_h}.
\]
Hence
\begin{equation}
\|\extension u-u_h\|_{H^1(\Omega_h)}\le C\inf_{v_h\in V_h}\|\extension u-v_h\|_{H^1(\Omega_h)}
+C\sup_{w_h\in V_h}\frac{|\delta_h(w_h)|}{\|w_h\|_{H^1(\Omega_h)}}.
\label{eq:i}
\end{equation}

The consistency term can be bounded as follows. Since $\extension u$ is in
$H^2(\RE^d)$ and $w_h\in H^1_0(\Omega_h)$, we can integrate by parts on
$\Omega_h$; moreover $\Delta\extension u=\Delta u=-f$ in $\Omega$. Therefore
\begin{equation}
\delta_h(w_h)=-\int_{\Omega_h}(\Delta\extension u+f)\,w_h\,\dx
=-\int_{\Omega_h\setminus\Omega}(\Delta\extension u+f)\,w_h\,\dx.
\label{eq:ii}
\end{equation}

We have $\Omega_h\setminus\Omega=\bigcup_{T\in\ThB}\Smt$.
Since $w_h=0$ on $\partial\Omega_h$, every point of $\Smt$ lies within
distance $O(h_T^2)$ of $\pOh$, and $|\Smt|\le Ch_T^{d+1}$,
Lemma~\ref{le:local} gives
\[
\aligned
\|w_h\|^2_{L^2(\Smt)}&\le|\Smt|\|w_h\|^2_{L^\infty(\Smt)}
\le Ch_T^{d+1}h_T^4|T|^{-1}\|\Grad w_h\|^2_{L^2(\patch T)}\\
&\le Ch_T^5\|\Grad w_h\|^2_{L^2(\patch T)},
\endaligned
\]
and summing over $T\in\mathcal{T}_h^B$,
\[
\|w_h\|_{L^2(\Omega_h\setminus\Omega)}\le Ch^{5/2}\|\Grad w_h\|_{L^2(\Omega_h)} .
\]
The Cauchy--Schwarz inequality in~\eqref{eq:ii}, together with~\eqref{eq:i},
concludes the proof.
\end{proof}

Putting together the results proved so far, we conclude this section by
stating the main result concerning the error estimate.

\begin{theorem}\label{th:Dfinal}
Let $k\ge1$, $t=\min(k,2)$, $s=\min(k,3/2)$, and let
$u\in H^1_0(\Omega)\cap H^{t+1}(\Omega)$ be the solution of \eqref{eq:poissongvar}
with $g=0$ and $u_h\in V_h$ the discrete solution of \eqref{eq:poissongvarh}
obtained using finite elements of polynomial degree $k$. Then
\[
\|u-u_h\|_{H^1(\Omega\cap\Omega_h)}\le\|\extension u-u_h\|_{H^1(\Omega_h)}
\le C\big(h^{s}\|u\|_{H^{t+1}(\Omega)}+h^{5/2}\|f\|_{L^2(\Omega_h\setminus\Omega)}\big).
\]
\end{theorem}

\begin{proof}
The first inequality holds since $\extension u=u$ in $\Omega$. By
\eqref{eq:Stein}, we get
$\|f+\Delta\extension u\|_{L^2(\Omega_h\setminus\Omega)}\le\|f\|_{L^2(\Omega_h\setminus\Omega)}
+C\|u\|_{H^2(\Omega)}$. For the best approximation term in \eqref{eq:strangD},
if $k\ge2$ we take $v_h=\tilde u^I$ and use \eqref{eq:pariopp}; if $k=1$ the
nodal interpolant of $\extension u$ belongs to $V_h$, because the boundary vertices lie on
$\partial\Omega$ where $\extension u=0$, and gives $Ch\|u\|_{H^2(\Omega)}$.
\end{proof}

\begin{remark}
The above bound gives a barrier on the error estimates in the energy norm
for higher-order polynomial approximation, independently of the smoothness of
$u$. The cap is aligned with the analogous result for the best approximation
stated in Remark~\ref{re:barrier}.
\label{re:barrier2}
\end{remark}

\section{Error estimate in $L^2$ with Dirichlet boundary conditions}
\label{se:DL2}

In this section we state the duality argument leading to the generalization of
the Aubin--Nitsche technique to our setting.

For simplicity, we assume that the domain $\Omega$ is strongly Lipschitz, so
that $\Oh$ is a Lipschitz polytope with uniformly bounded Lipschitz constant.

As usual, we need some regularity of the adjoint problem and we state in the
next lemma a useful result. Although we could not find it in this form in the
literature, it follows from classical results in harmonic analysis,
see~\cite{JerisonKenig1981,Verchota1984,Kenig1994,AmroucheMoussaoui2025,
JerisonKenig1995}.

\begin{lemma}\label{le:dualreg}
For $r\in L^2(\Oh)$ let $w(h)\in H^1_0(\Oh)$ be the solution of
\[
a_h(v,w(h))=(r,v)_{\Oh}\qquad\forall v\in H^1_0(\Oh).
\]
Then, with $C$ depending only on the Lipschitz constant of $\Omega$, we have
\begin{equation}\label{eq:apriori}
\|\Grad w(h)\|_{L^2(\OBpatch)}\le C\delta^{1/2}\|r\|_{L^2(\Oh)},
\end{equation}
where the definition of the layer $\OBpatch$ $(0<\delta\le\diam\Oh)$ was given
before Assumption~\ref{as:lunette}.

We also have $w(h)\in H^{3/2}(\Oh)$ with
$\|w(h)\|_{H^{3/2}(\Oh)}\le C\|r\|_{L^2(\Oh)}$.
\end{lemma}

\begin{theorem}
Let $u\in H^3(\Omega)$ be the solution of~\eqref{eq:poissongvar} for $g=0$ and
$u_h\in V_h$ the corresponding discrete solution of~\eqref{eq:poissongvarh}.
Then, for $k\ge2$,
\[
\|\extension u-u_h\|_{L^2(\Oh)}\le Ch^2(\|u\|_{H^3(\Omega)}+\|f\|_{L^2(\OhmO)}).
\]
\label{th:DL2}
\end{theorem}

\begin{proof}

In order to estimate $\|\extension u-u_h\|$ we add and subtract $\lenoir{u}$.
We consider the auxiliary problem: find $w(h)\in\Huoh$ such that
\[
a_h(v,w(h))=(\lenoir{u}-u_h,v)_\Oh\quad\forall v\in\Huoh.
\]
From Lemma~\ref{le:dualreg}, we have the a priori estimate
\[
\|w(h)\|_{H^{3/2}(\Oh)}\le C\|\lenoir{u}-u_h\|_{L^2(\Oh)}.
\]

A crucial ingredient of any duality argument is the Galerkin orthogonality
that in this setting holds up to an error arising from the geometry
perturbation. More precisely, for any $v_h\in V_h$, we have
\[
a_h(\lenoir{u}-u_h,v_h)=a_h(\lenoir{u},v_h)-(f,v_h)_\Oh.
\]
We analyze the first term; in the integrals we denote by
$\dxO$ and $\dxOh$ the volume elements, taking into account the change of
variable $\bfx_\Omega=\Phi_h(\bfx_\Oh)$. Hence
\[
\aligned
&\int_\Oh\Grad\lenoir{u}\cdot\Grad v_h\,\dxOh\\
&\quad =\int_\Omega(\Jac\Phi_h^\top\Grad u)\cdot(\Jac\Phi_h^\top\Grad
(v_h\circ\Phi_h^{-1}))|\det\Jac\Phi_h|^{-1}\,\dxO\\
&\quad =\int_\Omega\Grad u\cdot\Grad(v_h\circ\Phi_h^{-1})\,\dxO\\
&\qquad+\int_\Omega(\Grad u)^\top\left(\frac{\Jac\Phi_h(\Jac\Phi_h)^\top}
{|\det\Jac\Phi_h|}-\matId\right)\Grad(v_h\circ\Phi_h^{-1})\,\dxO\\
&\quad=(f,v_h\circ\Phi_h^{-1})_\Omega
+\int_\Omega(\Grad u)^\top\left(\frac{\Jac\Phi_h(\Jac\Phi_h)^\top}
{|\det\Jac\Phi_h|}-\matId\right)\Grad(v_h\circ\Phi_h^{-1})\,\dxO.
\endaligned
\]

We denote by $w(h)^I$ a quasi-interpolant of $w(h)$ --- for instance we can
take the Scott--Zhang construction~\cite{ScottZhang} --- so that we get

\begin{equation}
\aligned
\|\lenoir{u}-u_h\|_{L^2(\Oh)}^2&=a_h(\lenoir{u}-u_h,w(h))\\
&=a_h(\lenoir{u}-u_h,w(h)-w(h)^I)+\resto(h)\\
&\le C\|\lenoir{u}-u_h\|_{H^1(\Oh)}h^{1/2}\|w(h)\|_{H^{3/2}(\Oh)}
+|\resto(h)|,
\endaligned
\label{eq:AN}
\end{equation}
where
\begin{equation}
\aligned
\resto(h)&=(f,w(h)^I\circ\Phi_h^{-1})_\Omega-(f,w(h)^I)_\Oh\\
&\quad+\int_\Omega(\Grad u)^\top\left(\frac{\Jac\Phi_h(\Jac\Phi_h)^\top}
{|\det\Jac\Phi_h|}-\matId\right)\Grad(w(h)^I\circ\Phi_h^{-1})\,\dxO\\
&=\resto_1(h)+\resto_2(h).
\endaligned
\label{eq:r1r2}
\end{equation}

We estimate the two remainders. We have
\[
\resto_1(h)=\sum_{T\in\ThB}\left(\int_\Tt fw(h)^I\circ\Phi_h^{-1}\,\dxO
-\int_T fw(h)^I\,\dxO\right).
\]
We recall that $\OB$ is the region occupied by all the $T$ in $\ThB$:
\begin{align*}
\resto_1(h)&=\sum_{T\in\ThB}\bigg(\int_T f(w(h)^I\circ\Phi_h^{-1}-w(h)^I)\,\dxO\\
&\quad+\int_{\Spt} fw(h)^I\circ\Phi_h^{-1}\,\dxO
-\int_{\Smt} fw(h)^I\circ\Phi_h^{-1}\,\dxO\bigg)\\
&\le\|f\|_{L^2(\OB)}\left(\sum_{T\in\ThB}\|w(h)^I\|^2_{L^2(T)}
\|\Phi_h^{-1}-\mathbf{I}\|^2_{L^\infty(T)}\right)^{1/2}\\
&\quad+\Big|\sum_{T\in\ThB}\int_{\Spt}
fw(h)^I\circ\Phi_h^{-1}\,\dxO\Big|
+\Big|\sum_{T\in\ThB}\int_{\Smt}
fw(h)^I\circ\Phi_h^{-1}\,\dxO\Big|\\
&\le Ch^2\|f\|_{L^2(\OB)}\|w(h)^I\|_{L^2(\OB)}\\
&\quad+\Big|\sum_{T\in\ThB}\int_{\Spt}
fw(h)^I\circ\Phi_h^{-1}\,\dxO\Big|
+\Big|\sum_{T\in\ThB}\int_{\Smt}
fw(h)^I\circ\Phi_h^{-1}\,\dxO\Big|
\end{align*}

The first of the last two terms can be estimated as
\[
\aligned
&\Big|\sum_{T\in\ThB}\int_{\Spt}fw(h)^I\circ\Phi_h^{-1}\,\dxO\Big|
\le\Big|\sum_{T\in\ThB}\int_{\Phi_h^{-1}(\Spt)}
\lenoir{f}w(h)^I|\det\Jac\Phi_h|\,\dxOh\Big|\\
&\quad\le C\|f\|_{L^2(\OmOh)}
\left(\sum_{T\in\ThB}|\Phi_h^{-1}(\Spt)|
\|w(h)^I\|^2_{L^\infty(\Phi_h^{-1}(\Spt))}\right)^{1/2}\\
&\quad\le C\|f\|_{L^2(\OmOh)}
\left(\sum_{T\in\ThB}|\Phi_h^{-1}(\Spt)|
h^4\|\Grad w(h)^I\|^2_{L^\infty(\Phi_h^{-1}(\Spt))}\right)^{1/2}\\
&\quad\le C\|f\|_{L^2(\OmOh)}
\left(\sum_{T\in\ThB}
h^4\|\Grad w(h)^I\|^2_{L^2(\Phi_h^{-1}(\Spt))}\right)^{1/2}\\
&\quad\le Ch^2\|f\|_{L^2(\OmOh)}\|\Grad w(h)^I\|_{L^2(\OB)}
\le Ch^2\|f\|_{L^2(\OmOh)}\|w(h)\|_{H^1(\OB)},
\endaligned
\]
where we used~\eqref{eq:lenoir} and that $\|w(h)^I\|_{L^\infty}$ is bounded by
$h^2\|\Grad w(h)^I\|_{L^\infty}$ because of the boundary conditions and the
inverse estimate
$|\omega|\|w_h\|^2_{L^\infty(\omega)}\le C\|w_h\|^2_{L^2(\omega)}$ for all
$w_h\in V_h$.

The term involving the integrals over $\Smt$ can be treated with
similar arguments, using the transformation $\Phi_h(\Smt)$,
thus obtaining
\[
\Big|\sum_{T\in\ThB}\int_{\Smt}fw(h)^I\circ\Phi_h^{-1}\,\dxO\Big|\\
\le Ch^2\|f\|_{L^2(\OhmO)}\|w(h)\|_{H^1(\OB)}.
\]

We then have
\[
|\resto_1(h)|\le Ch^2\|f\|_{L^2(\Obar)}\|\lenoir{u}-u_h\|_{L^2(\Oh)}.
\]

For the second term in~\eqref{eq:r1r2}, using the results of
Section~\ref{se:geom}, in particular the bounds in~\eqref{eq:cov} and
Lemma~\ref{le:local}, the stability of the Scott--Zhang quasi-interpolant,
and~\eqref{eq:apriori}, we have the following estimate for $\delta\le Ch$:
\[
\aligned
|\resto_2(h)|&\le Ch\|\Grad u\|_{L^2(\Phi_h(\OB))}
\|\Grad(w(h)^I\circ\Phi_h^{-1})\|_{L^2(\Phi_h(\OB))}\\
&\le Ch|\Phi_h(\OB)|^{1/2}\,\|\Grad u\|_{L^\infty(\Omega)}
\|\Grad w(h)^I\|_{L^2(\OB)}\\
&\le Ch^{3/2}\|u\|_{H^3(\Omega)}\|\Grad w(h)^I\|_{L^2(\OB)}\\
&\le Ch^{3/2}\|u\|_{H^3(\Omega)}\|\Grad w(h)\|_{L^2(\OBpatch)}\\
&\le Ch^2\|u\|_{H^3(\Omega)}\|\lenoir{u}-u_h\|_{L^2(\Oh)}.
\endaligned
\]

We conclude the bound for $\|\lenoir{u}-u_h\|$ by finally
estimating~\eqref{eq:AN}. We use the bounds of $\lenoir{u}-u_h$ and
$w(h)-w(h)^I$ in the seminorm of $H^1(\Oh)$ together with the one for
$\resto(h)$, taking into account that $\|f\|_{L^2(\Obar)}\le
C(\|u\|_{H^3(\Omega)}+\|f\|_{L^2(\OhmO)})$, so that we get
\[
\aligned
\|\lenoir{u}-u_h\|_{L^2(\Oh)}&\le Ch^{1/2}\|\lenoir{u}-u_h\|_{H^1(\Oh)}+
Ch^2(\|u\|_{H^3(\Omega)}+\|f\|_{L^2(\OhmO)})\\
&\le Ch^2(\|u\|_{H^3(\Omega)}+\|f\|_{L^2(\OhmO)}).
\endaligned
\]

To conclude the proof we estimate $\|\extension u-\lenoir{u}\|$. Since
$\extension u$ and $\lenoir{u}$ differ only on $\OB$ we take $y\in\OB$
and get
\[
\aligned
|\extension u(y)-\lenoir{u}(y)|&=
|\extension u(y)-\extension u(\Phi_h(y))|\\
&\le\|\Grad\extension u\|_{L^\infty(\R^d)}\,|y-\Phi_h(y)|
\le Ch^2\|u\|_{H^3(\Omega)}.
\endaligned
\]
Integrating over $\OB$ and taking into account that its measure is $O(h)$,
we conclude the proof with the missing bound
\[
\|\extension u-\lenoir{u}\|_{L^2(\Oh)}\le Ch^{5/2}\|u\|_{H^3(\Omega)}.
\]

\end{proof}

\begin{remark}
The estimate of the previous theorem gives a barrier of $O(h^2)$ on the rate
of convergence in $L^2$ of higher-order approximation, independently of the
smoothness of the solution of the original and the dual problems. The cap is
worse than for the best approximation estimate by $O(h^{1/2})$.
\label{re:barrier3}
\end{remark}

\section{Error analysis in the case of Neumann boundary conditions}
\label{se:neumann}

In this section we sketch what is different with respect to the previous
analysis in the case of Neumann boundary conditions.

We are dealing with the following problem: find $u$ such that
\begin{equation}
\aligned
&-\Delta u+u=f&&\text{in }\Omega\\
&\Grad u\cdot\n=g_N&&\text{on }\partial\Omega.
\endaligned
\label{eq:poissonneumann}
\end{equation}
We add the reaction term in order to avoid compatibility conditions. With
$V=\Hu$ a standard variational formulation reads: find $u\in V$ such that
\begin{equation}
b(u,v)=(f,v)_\Omega+\dual{g_N}{v}\quad\forall v\in V,
\label{eq:poissonneumannvar}
\end{equation}
where
\[
b(u,v)=(\Grad u,\Grad v)_\Omega+(u,v)_\Omega
\]
and $\dual{\cdot}{\cdot}$ denotes the duality pairing between
$H^{1/2}(\partial\Omega)$ and its dual $H^{-1/2}(\partial\Omega)$. As for the
Dirichlet case, we assume suitable regularity of our data so that the
solution $u$ belongs to $H^3(\Omega)$ and the duality pairing can be written
as an integral on the boundary
\[
\dual{g_N}{v}=\int_{\partial\Omega}g_Nv\,ds.
\]

Without repeating the discussion and the definitions detailed in
Section~\ref{se:setting}, given a finite element space $V_h$ conforming in
$H^1(\Oh)$, the discretization of~\eqref{eq:poissonneumannvar} reads: find
$u_h\in V_h$ such that
\begin{equation}
b_h(u,v)=(f_h,v)_\Oh+\int_{\pOh}g_{N,h}v\,ds\quad\forall v\in V_h,
\label{eq:poissonneumannvarh}
\end{equation}
where $b_h(u,v)$ has the same expression as $b(u,v)$, on $\Oh$, and where
a suitable approximation $g_{N,h}$ of $g_N$ should be considered.

Since the aim of this section is to highlight the main differences between
Dirichlet and Neumann boundary conditions, we start by considering the
homogeneous case $g_N=0$ with the natural choice $g_{N,h}=0$. As in the
previous case, we consider $f_h=f$.

The first important difference from the Dirichlet case is that, since the
boundary conditions are not imposed in the spaces $V$ and $V_h$, the best
approximation error $\inf_{v\in V_h}\|\extension{u}-v\|$ is always optimal
with respect to the $H^1$ and $L^2$ norms.

The energy norm estimate of the error can be obtained by using arguments
similar to those of Strang's theorem, which give the following result.

\begin{lemma}
Let $u\in V$ be the solution of~\eqref{eq:poissonneumannvar} with $g_N=0$ and
$u_h\in V_h$ the corresponding solution of~\eqref{eq:poissonneumannvarh}. Then
\begin{equation}
\|\extension{u}-u_h\|_{\Huh}\le C\inf_{v_h\in V_h}\|\extension{u}-v_h\|_{\Huh}
+\sup_{w_h\in V_h}\frac{|\delta_h(w_h)|}{\|w_h\|_{\Huh}}
\label{eq:ceaN}
\end{equation}
with
\[
\delta_h(w_h)=b_h(\extension{u},w_h)-(f,w_h)_{\Oh}.
\]
\label{le:ceaN}
\end{lemma}
Since the first term in~\eqref{eq:ceaN} is optimal, a limitation on the error
estimates can only come from the consistency term (the second one). We have
the following estimate analogous to the results of Theorem~\ref{th:Dfinal}.
\begin{theorem}
Let $u\in V$ be the solution of~\eqref{eq:poissonneumannvar} with $g_N=0$ and
$u_h\in V_h$ the solution of~\eqref{eq:poissonneumannvarh} obtained using
finite elements of polynomial degree $k\ge1$. If $u$ belongs to $H^3(\Omega)$,
then
\[
\|\extension u-u_h\|_{H^1(\Oh)}\le
C(h^s\|u\|_{H^3(\Omega)}+h^2\|f\|_{L^2(\OhmO)}),
\]
with $s=\min(k,3/2)$.
\label{th:Nfinal}
\end{theorem}
\begin{proof}

The consistency term in~\eqref{eq:ceaN} can be estimated by adding and
subtracting $\lenoir{u}$ as follows:
\[
b_h(\extension{u},w_h)=b_h(\extension{u}-\lenoir{u},w_h)+b_h(\lenoir{u},w_h).
\]

It turns out that the difference $b_h(\lenoir{u},w_h)-(f,w_h)_{\Oh}$ can be
estimated as in the proof of Theorem~\ref{th:DL2} with the natural
modifications coming from the zeroth-order term, leading to
\[
\aligned
&b_h(\lenoir{u},w_h)-(f,w_h)_{\Oh}=
\int_\Omega(\Grad u)^\top
\left(\frac{\Jac\Phi_h\Jac\Phi_h^\top}{|\det\Jac\Phi_h|}
-\matId\right)\Grad w_h\circ\Phi_h^{-1}\,\dxO\\
&\quad+\int_\Omega u w_h\circ\Phi_h^{-1}\big(|\det\Jac\Phi_h|^{-1}-1\big)\,\dxO
+(f,w_h\circ\Phi_h^{-1})_\Omega-(f,w_h)_{\Oh}.
\endaligned
\]
As in the previous case we get
\[
|b_h(\lenoir{u},w_h)-(f,w_h)_{\Oh}|\le
C(h^{3/2}\|u\|_{H^3(\Omega)}+h^2\|f\|_{L^2(\OhmO)})\|w_h\|_{\Huh}.
\]

The last term to be considered is $b_h(\extension{u}-\lenoir{u},w_h)$ which
can be estimated by using the following bound
\[
\|\extension{u}-\lenoir{u}\|_{\Huh}\le Ch^{3/2}\|u\|_{H^3(\Omega)};
\]
this bound is analogous to what we obtained at the end of the proof of
Theorem~\ref{th:DL2} and originates from the fact that the left-hand side
vanishes outside $\OB$, together with the usual estimates of
Proposition~\ref{pr:phih}.

\end{proof}

We conclude the discussion about Neumann boundary conditions by studying the
error in $L^2$, in analogy with what was done in Section~\ref{se:DL2} in the
Dirichlet case.

\begin{theorem}
Let $u\in H^3(\Omega)$ be the solution of~\eqref{eq:poissonneumannvar} for
$g_N=0$ and $u_h\in V_h$ the corresponding discrete solution
of~\eqref{eq:poissonneumannvarh}. Assume moreover that $f$ belongs to
$H^1(\Obar)$.  Then, for $k\ge2$,
\[
\|\extension u-u_h\|_{L^2(\Oh)}\le Ch^2(\|u\|_{H^3(\Omega)}+\|f\|_{H^1(\OhmO)}).
\]
\label{th:NL2}
\end{theorem}

\begin{proof} \emph{(Sketch)}.
We highlight the differences from the proof of Theorem~\ref{th:DL2} when
Neumann boundary conditions are used instead of Dirichlet.
The first difference, making the analysis easier, is that, since we do not
need to be conforming in $\Huo$ but in $\Hu$ only, the dual problem does not
need to consider $\lenoir{u}-u_h$ but can work directly with
$\extension u-u_h$. In this case the dual problem reads: find $w(h)\in\Huh$
such that
\[
b_h(v,w(h))=(\extension u-u_h,v)_{\Oh}\quad\forall v\in\Huh.
\]
By proceeding along the lines of Theorem~\ref{th:DL2}, we have
\[
\|\extension u-u_h\|^2_{L^2(\Oh)}=
b_h(\extension u-u_h,w(h)-w(h)^I)+\delta_h(w(h)^I),
\]
with
\[
\delta_h(v)=b_h(\extension u,v)-(f,v)_{\Oh}.
\]
The expression $\delta_h(w(h)^I)$ is the analogue of $\resto(h)$ in the
proof of the Dirichlet case and we split it similarly as the sum of
$\resto_1(h)$ and $\resto_2(h)$.

The estimation of $\resto_2(h)$ can proceed verbatim as in that proof, since
the boundary conditions were not used there. We get
\[
|\resto_2(h)|\le Ch^2\|u\|_{H^3(\Omega)}\|\extension u-u_h\|_{L^2(\Oh)}.
\]

On the other hand, the first term $\resto_1(h)$ requires an ad hoc treatment.
We will use the following bound, valid for $\delta\in(0,\diam\Oh]$ whenever
$\Omega$ is a strongly Lipschitz domain:
\[
\|v\|_{L^2(\OBpatch)}\le C\delta^{1/2}\|v\|_{\Huh}\quad\forall v\in\Huh.
\]
We will also use the following localized bound that comes from the extra
regularity of $f$:
\[
\|f\|_{L^2(\OmOh)}+\|f\|_{L^2(\OhmO)}\le Ch\|f\|_{H^1(\Obar)}.
\]
With these estimates, the first remainder can be estimated as
\[
|\resto_1(h)|\le Ch^2\|f\|_{H^1(\Obar)}\|\extension u-u_h\|_{L^2(\Oh)}.
\]

The final result is obtained by observing that from the problem in strong form
we have that $\|f\|_{H^1(\Obar)}$ is bounded by
$\|u\|_{H^3(\Omega)}+\|f\|_{H^1(\OhmO)}$.

\end{proof}

\begin{remark}
Theorems~\ref{th:Nfinal} and~\ref{th:NL2} show barriers similar to those of
the Dirichlet case. Namely, the energy norm error is capped by $O(h^3/2)$ and
the $L^2$ error by $O(h^2)$ no matter how smooth the solution is, and
independently of the degree $k\ge2$.
\label{re:barrier4}
\end{remark}

\section{Non homogeneous boundary conditions}
\label{se:non0}

The extension of the previous results, both in the Dirichlet and Neumann
cases, does not add significant technicalities. Actually, the major
difficulties arise from the definition of the approximation of $g$
in~\eqref{eq:poissongvar} and of $g_N$ in~\eqref{eq:poissonneumannvar}.

If these definitions are performed without adding errors of lower order, the
caps of Remarks~\ref{re:barrier2}, \ref{re:barrier3}, and~\ref{re:barrier4}
remain unchanged.

We sketch the arguments in the case of Dirichlet boundary conditions. The
solution of~\eqref{eq:poissongvar} is seen as $u=u_0+u_g$. It is natural to
choose as $g_h$ a discretization of the function $\lenoir{g}=g\circ\Phi_h$.
Then the estimate of $u_0-u_{0,h}$ follows what we have already done, with an
additional term coming from the source term that depends on
$\|u_g-u_{\lenoir{g}}\|$. This additional term relies on the bound
$\|\Id-\Phi_h\|_{\infty}=O(h^2)$, obtained from Proposition~\ref{pr:phih}. It
follows that the existing caps ($O(h^{3/2})$ and $O(h^2)$ in $H^1$ and $L^2$,
respectively) are not affected by the additional consistency term.
In practice, $g_h$ can be taken as a projection of $\lenoir{g}$ so that no
quadrature error is added to the computation of $u_h$.

The non-homogeneous Neumann case involves the estimate of the mismatch of the
boundary terms, namely
\[
\int_{\partial\Omega}g_Nv\,ds\quad(v\in V)
\]
and
\[
\int_{\pOh}g_{N,h}v\,ds\quad(v\in V_h).
\]
Here, again, we can use the mapping $\Phi_h$ and estimate the following
difference when we consider $g_{N,h}$ as a suitable projection of
$\lenoir{(g_N)}$:
\[
\int_{\partial\Omega}g_Nv_h\circ\Phi_h^{-1}\,ds-
\int_{\pOh}g_{N,h}v_h\,ds\quad(v_h\in V_h).
\]
The estimate of this last term, when $g_{N,h}$ is the interpolant of
$\lenoir{(g_N)}$, can be performed as in~\cite[Lemma on p.~579]{Lenoir}.

\section{Numerical results}
\label{se:num}

In this section we collect some numerical results that complement the
theoretical investigations presented before. Our codes are written in Python
with the legacy FEniCS library. 

Our results are computed on four domains, all satisfying
Assumption~\ref{as:domain} and sketched in Figure~\ref{fg:domains}:
two in two dimensions and two in three dimensions.
The circle and the sphere are
centered at the origin and have radius one. The other two domains are the
box $(0,1)^d$ minus a portion of the ball centered at the origin and of radius
$1/2$.
The domains are approximated by a sequence of polytopes that are triangulated
using Gmsh~\cite{Gmsh} in such a way that Assumption~\ref{as:lunette} is
satisfied.
The domains are chosen so that their boundaries have both convex and concave
components.

\begin{figure}
\begin{center}
\includegraphics[scale=.65]{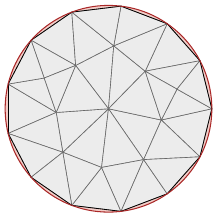}\hspace{.2cm}
\includegraphics[scale=.65]{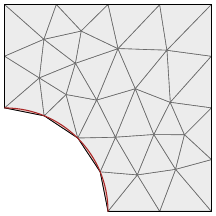}\hspace{.2cm}
\includegraphics[scale=.65]{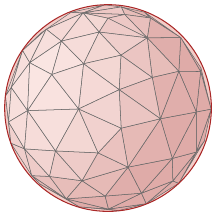}\hspace{.2cm}
\includegraphics[scale=.65]{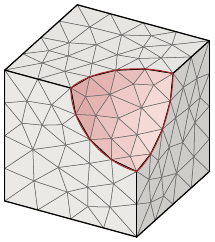}
\caption{Our domains $\Omega$ and a sample of corresponding computational
domains $\Oh$.}
\label{fg:domains}
\end{center}
\end{figure}

For all our tests we use finite elements of degree $k=1,2,3$.

In Table~\ref{tb:rates} we recall the theoretical bounds that we proved in the
previous sections.

\begin{table}
\begin{center}
\footnotesize
\begin{tabular}{c|cc|cc|c}
\hline\hline\\[-8pt]
 & \multicolumn{2}{c|}{best approximation} & \multicolumn{2}{c|}{error $u-u_h$} & \\
$k$ & $L^2$ & $H^1$ & $L^2$ & $H^1$ & regularity\\[3pt]
\hline
\multicolumn{6}{c}{essential (Dirichlet) condition}\\
\hline
$1$ & $h^{2}$   & $h$       & $h^{2}$ & $h$       & $u\in H^{2}(\Omega)$\\
$2$ & $h^{5/2}$ & $h^{3/2}$ & $h^{2}$ & $h^{3/2}$ & $u\in H^{3}(\Omega)$\\
$3$ & $h^{5/2}$ & $h^{3/2}$ & $h^{2}$ & $h^{3/2}$ & $u\in H^{4}(\Omega)$\\
\hline
\multicolumn{6}{c}{natural (Neumann) condition}\\
\hline
$1$ & $h^{2}$ & $h$     & $h^{2}$ & $h$       & $u\in H^{2}(\Omega)$\\
$2$ & $h^{3}$ & $h^{2}$ & $h^{2}$ & $h^{3/2}$ & $u\in H^{3}(\Omega)$\\
$3$ & $h^{4}$ & $h^{3}$ & $h^{2}$ & $h^{3/2}$ & $u\in H^{4}(\Omega)$\\
\hline
\end{tabular}
\par\vspace{10pt}
\caption{Theoretical convergence rates in $h$ for the straight-sided $P_k$
approximation of the homogeneous Dirichlet and Neumann problems on $\Omega_h$.}
\label{tb:rates}
\end{center}
\end{table}

\begin{example}
\emph{Best approximation.}

We start with the best approximation in the case of Dirichlet boundary
conditions: given a function $u$ defined in all $\RE^d$ and vanishing on
$\partial\Omega$ (or part of it), we compute its $L^2(\Oh)$ projection onto
the finite element space of functions vanishing on $\partial\Oh$ (or on the
corresponding part of $\partial\Oh$). The function $u$ is chosen as follows:
\[
\aligned
&u(x,y)=(x^2+y^2-1)e^{x+y}&&\text{for the circle}\\
&u(x,y)=(x^2+y^2-0.25)e^{x+y}&&\text{for the square minus the circle}\\
&u(x,y,z)=(x^2+y^2+z^2-1)e^{x+y+z}&&\text{for the sphere}\\
&u(x,y,z)=(x^2+y^2+z^2-0.25)e^x&&\text{for the cube minus the sphere}
\endaligned
\]
All these functions vanish on the curved parts of the boundaries of their
domains.

The results are shown in Tables~\ref{tb:proj_circle}, \ref{tb:proj_4hole},
\ref{tb:proj_sphere}, and~\ref{tb:proj_4hole3D}, respectively, and confirm the
barrier stated in Remark~\ref{re:barrier}.

\begin{table}
\setlength{\tabcolsep}{3pt}
\noindent%
\begin{minipage}[t]{\dimexpr(\linewidth-1\dimexpr0.03\linewidth\relax)/2\relax}
\setlength{\abovecaptionskip}{0pt}\setlength{\belowcaptionskip}{0pt}
\footnotesize
\sbox0{%
\begin{tabular}[t]{rr|cr|cr}
\hline\hline\\[-8pt]
Nelements&Ndofs&\multicolumn{1}{r}{$\|u-u_h\|_{L^2}$}&rate&\multicolumn{1}{r}{$\vert u-u_h\vert_{H^1}$}&rate\\[3pt]
\hline
\multicolumn{6}{c}{$k= 1$}\\
\hline
8424 & 4314 & $4.815\times 10^{-4}$ & -- & $1.129\times 10^{-1}$ & -- \\
33419 & 16912 & $1.127\times 10^{-4}$ & 2.13 & $5.549\times 10^{-2}$ & 1.04 \\
139743 & 70275 & $2.692\times 10^{-5}$ & 2.01 & $2.758\times 10^{-2}$ & 0.98 \\
548805 & 275208 & $6.601\times 10^{-6}$ & 2.06 & $1.376\times 10^{-2}$ & 1.02 \\
2146423 & 1074821 & $1.634\times 10^{-6}$ & 2.05 & $6.883\times 10^{-3}$ & 1.02 \\
8527920 & 4267178 & $4.065\times 10^{-7}$ & 2.02 & $3.441\times 10^{-3}$ & 1.01 \\
\hline
\multicolumn{6}{c}{$k= 2$}\\
\hline
8424 & 17051 & $5.319\times 10^{-5}$ & -- & $1.335\times 10^{-2}$ & -- \\
33419 & 67242 & $9.411\times 10^{-6}$ & 2.52 & $4.773\times 10^{-3}$ & 1.50 \\
139743 & 280292 & $1.662\times 10^{-6}$ & 2.43 & $1.681\times 10^{-3}$ & 1.46 \\
548805 & 1099220 & $2.960\times 10^{-7}$ & 2.53 & $5.944\times 10^{-4}$ & 1.52 \\
2146423 & 4296064 & $5.218\times 10^{-8}$ & 2.55 & $2.105\times 10^{-4}$ & 1.52 \\
8527920 & 17062275 & $9.222\times 10^{-9}$ & 2.51 & $7.439\times 10^{-5}$ & 1.51 \\
\hline
\multicolumn{6}{c}{$k= 3$}\\
\hline
8424 & 38212 & $3.586\times 10^{-5}$ & -- & $1.578\times 10^{-2}$ & -- \\
33419 & 150991 & $6.404\times 10^{-6}$ & 2.51 & $5.601\times 10^{-3}$ & 1.51 \\
139743 & 630052 & $1.132\times 10^{-6}$ & 2.43 & $1.985\times 10^{-3}$ & 1.45 \\
548805 & 2472037 & $2.019\times 10^{-7}$ & 2.52 & $6.992\times 10^{-4}$ & 1.53 \\
2146423 & 9663730 & $3.561\times 10^{-8}$ & 2.55 & $2.480\times 10^{-4}$ & 1.52 \\
8527920 & 38385292 & $6.294\times 10^{-9}$ & 2.51 & $8.772\times 10^{-5}$ & 1.51 \\
\hline
\end{tabular}%
}%
{\centering\ifdim\wd0>\linewidth \resizebox{\linewidth}{!}{\usebox0}\else\usebox0\fi\par}
\par\vspace{10pt}
\caption{Convergence rate for the $L^2$ projection on the circle.}
\label{tb:proj_circle}
\end{minipage}%
\hfill
\begin{minipage}[t]{\dimexpr(\linewidth-1\dimexpr0.03\linewidth\relax)/2\relax}
\setlength{\abovecaptionskip}{0pt}\setlength{\belowcaptionskip}{0pt}
\footnotesize
\sbox0{%
\begin{tabular}[t]{rr|cr|cr}
\hline\hline\\[-8pt]
Nelements&Ndofs&\multicolumn{1}{r}{$\|u-u_h\|_{L^2}$}&rate&\multicolumn{1}{r}{$\vert u-u_h\vert_{H^1}$}&rate\\[3pt]
\hline
\multicolumn{6}{c}{$k= 1$}\\
\hline
8785 & 4515 & $2.274\times 10^{-4}$ & -- & $1.108\times 10^{-1}$ & -- \\
35037 & 17762 & $5.348\times 10^{-5}$ & 2.11 & $5.438\times 10^{-2}$ & 1.04 \\
139544 & 70258 & $1.299\times 10^{-5}$ & 2.06 & $2.708\times 10^{-2}$ & 1.01 \\
557133 & 279537 & $3.215\times 10^{-6}$ & 2.02 & $1.353\times 10^{-2}$ & 1.00 \\
2230087 & 1116983 & $7.968\times 10^{-7}$ & 2.01 & $6.755\times 10^{-3}$ & 1.00 \\
8919855 & 4463805 & $1.985\times 10^{-7}$ & 2.01 & $3.376\times 10^{-3}$ & 1.00 \\
\hline
\multicolumn{6}{c}{$k= 2$}\\
\hline
8785 & 17814 & $3.076\times 10^{-6}$ & -- & $1.623\times 10^{-3}$ & -- \\
35037 & 70560 & $5.440\times 10^{-7}$ & 2.52 & $5.638\times 10^{-4}$ & 1.54 \\
139544 & 280059 & $9.610\times 10^{-8}$ & 2.51 & $1.955\times 10^{-4}$ & 1.54 \\
557133 & 1116206 & $1.684\times 10^{-8}$ & 2.52 & $6.919\times 10^{-5}$ & 1.50 \\
2230087 & 4464052 & $3.004\times 10^{-9}$ & 2.49 & $2.426\times 10^{-5}$ & 1.51 \\
8919855 & 17847464 & $5.303\times 10^{-10}$ & 2.50 & $8.584\times 10^{-6}$ & 1.50 \\
\hline
\multicolumn{6}{c}{$k= 3$}\\
\hline
8785 & 39898 & $1.994\times 10^{-6}$ & -- & $1.808\times 10^{-3}$ & -- \\
35037 & 158395 & $3.629\times 10^{-7}$ & 2.47 & $6.468\times 10^{-4}$ & 1.49 \\
139544 & 629404 & $6.492\times 10^{-8}$ & 2.49 & $2.263\times 10^{-4}$ & 1.52 \\
557133 & 2510008 & $1.145\times 10^{-8}$ & 2.51 & $8.087\times 10^{-5}$ & 1.49 \\
2230087 & 10041208 & $2.045\times 10^{-9}$ & 2.48 & $2.848\times 10^{-5}$ & 1.51 \\
8919855 & 40150978 & $3.618\times 10^{-10}$ & 2.50 & $1.008\times 10^{-5}$ & 1.50 \\
\hline
\end{tabular}%
}%
{\centering\ifdim\wd0>\linewidth \resizebox{\linewidth}{!}{\usebox0}\else\usebox0\fi\par}
\par\vspace{10pt}
\caption{Convergence rate for the $L^2$ projection on the square minus a circle geometry.}
\label{tb:proj_4hole}
\end{minipage}%
\par
\par\vspace{1.5em}
\noindent%
\begin{minipage}[t]{\dimexpr(\linewidth-1\dimexpr0.03\linewidth\relax)/2\relax}
\setlength{\abovecaptionskip}{0pt}\setlength{\belowcaptionskip}{0pt}
\footnotesize
\sbox0{%
\begin{tabular}[t]{rr|cr|cr}
\hline\hline\\[-8pt]
Nelements&Ndofs&\multicolumn{1}{r}{$\|u-u_h\|_{L^2}$}&rate&\multicolumn{1}{r}{$\vert u-u_h\vert_{H^1}$}&rate\\[3pt]
\hline
\multicolumn{6}{c}{$k= 1$}\\
\hline
5430 & 1259 & $4.870\times 10^{-2}$ & -- & $1.364\times 10^{0}$ & -- \\
40105 & 7825 & $1.117\times 10^{-2}$ & 2.42 & $6.730\times 10^{-1}$ & 1.16 \\
320817 & 56747 & $2.384\times 10^{-3}$ & 2.34 & $3.213\times 10^{-1}$ & 1.12 \\
2473741 & 416162 & $5.186\times 10^{-4}$ & 2.30 & $1.539\times 10^{-1}$ & 1.11 \\
19491401 & 3191093 & $1.185\times 10^{-4}$ & 2.17 & $7.531\times 10^{-2}$ & 1.05 \\
154257125 & 24854019 & $2.788\times 10^{-5}$ & 2.11 & $3.715\times 10^{-2}$ & 1.03 \\
\hline
\multicolumn{6}{c}{$k= 2$}\\
\hline
728 & 1337 & $4.087\times 10^{-2}$ & -- & $8.112\times 10^{-1}$ & -- \\
5430 & 8648 & $7.176\times 10^{-3}$ & 2.80 & $2.881\times 10^{-1}$ & 1.66 \\
40105 & 58448 & $1.288\times 10^{-3}$ & 2.70 & $1.051\times 10^{-1}$ & 1.58 \\
320817 & 445238 & $2.194\times 10^{-4}$ & 2.62 & $3.663\times 10^{-2}$ & 1.56 \\
2473741 & 3349542 & $3.813\times 10^{-5}$ & 2.60 & $1.287\times 10^{-2}$ & 1.55 \\
19491401 & 26047504 & $6.672\times 10^{-6}$ & 2.55 & $4.537\times 10^{-3}$ & 1.53 \\
\hline
\multicolumn{6}{c}{$k= 3$}\\
\hline
271 & 1634 & $6.034\times 10^{-2}$ & -- & $1.489\times 10^{0}$ & -- \\
1467 & 7948 & $1.328\times 10^{-2}$ & 2.87 & $5.790\times 10^{-1}$ & 1.79 \\
10790 & 53640 & $2.388\times 10^{-3}$ & 2.70 & $2.135\times 10^{-1}$ & 1.57 \\
80074 & 378740 & $4.495\times 10^{-4}$ & 2.56 & $7.838\times 10^{-2}$ & 1.54 \\
622920 & 2867341 & $8.342\times 10^{-5}$ & 2.50 & $2.865\times 10^{-2}$ & 1.49 \\
4936952 & 22416570 & $1.475\times 10^{-5}$ & 2.53 & $1.012\times 10^{-2}$ & 1.52 \\
\hline
\end{tabular}%
}%
{\centering\ifdim\wd0>\linewidth \resizebox{\linewidth}{!}{\usebox0}\else\usebox0\fi\par}
\par\vspace{10pt}
\caption{Convergence rate for the $L^2$ projection on the sphere.}
\label{tb:proj_sphere}
\end{minipage}%
\hfill
\begin{minipage}[t]{\dimexpr(\linewidth-1\dimexpr0.03\linewidth\relax)/2\relax}
\setlength{\abovecaptionskip}{0pt}\setlength{\belowcaptionskip}{0pt}
\footnotesize
\sbox0{%
\begin{tabular}[t]{rr|cr|cr}
\hline\hline\\[-8pt]
Nelements&Ndofs&\multicolumn{1}{r}{$\|u-u_h\|_{L^2}$}&rate&\multicolumn{1}{r}{$\vert u-u_h\vert_{H^1}$}&rate\\[3pt]
\hline
\multicolumn{6}{c}{$k= 1$}\\
\hline
5243 & 1321 & $8.486\times 10^{-3}$ & -- & $3.964\times 10^{-1}$ & -- \\
36986 & 7635 & $1.868\times 10^{-3}$ & 2.59 & $1.960\times 10^{-1}$ & 1.20 \\
282071 & 51649 & $4.134\times 10^{-4}$ & 2.37 & $9.651\times 10^{-2}$ & 1.11 \\
2191566 & 374805 & $9.594\times 10^{-5}$ & 2.21 & $4.786\times 10^{-2}$ & 1.06 \\
17352530 & 2864003 & $2.275\times 10^{-5}$ & 2.12 & $2.374\times 10^{-2}$ & 1.03 \\
138171918 & 22340094 & $5.489\times 10^{-6}$ & 2.08 & $1.180\times 10^{-2}$ & 1.02 \\
\hline
\multicolumn{6}{c}{$k= 2$}\\
\hline
1012 & 1935 & $1.600\times 10^{-3}$ & -- & $6.327\times 10^{-2}$ & -- \\
5243 & 8756 & $2.735\times 10^{-4}$ & 3.51 & $2.011\times 10^{-2}$ & 2.28 \\
36986 & 55535 & $4.515\times 10^{-5}$ & 2.93 & $6.432\times 10^{-3}$ & 1.85 \\
282071 & 398290 & $7.478\times 10^{-6}$ & 2.74 & $2.082\times 10^{-3}$ & 1.72 \\
2191566 & 2991492 & $1.306\times 10^{-6}$ & 2.60 & $7.166\times 10^{-4}$ & 1.59 \\
17352530 & 23280977 & $2.345\times 10^{-7}$ & 2.51 & $2.545\times 10^{-4}$ & 1.51 \\
\hline
\multicolumn{6}{c}{$k= 3$}\\
\hline
364 & 2210 & $1.347\times 10^{-3}$ & -- & $8.584\times 10^{-2}$ & -- \\
1544 & 8652 & $4.354\times 10^{-4}$ & 2.48 & $3.517\times 10^{-2}$ & 1.96 \\
9665 & 49272 & $9.165\times 10^{-5}$ & 2.69 & $1.340\times 10^{-2}$ & 1.66 \\
72571 & 348698 & $1.719\times 10^{-5}$ & 2.57 & $4.657\times 10^{-3}$ & 1.62 \\
554579 & 2573625 & $2.842\times 10^{-6}$ & 2.70 & $1.559\times 10^{-3}$ & 1.64 \\
4366891 & 19913036 & $5.127\times 10^{-7}$ & 2.51 & $5.604\times 10^{-4}$ & 1.50 \\
\hline
\end{tabular}%
}%
{\centering\ifdim\wd0>\linewidth \resizebox{\linewidth}{!}{\usebox0}\else\usebox0\fi\par}
\par\vspace{10pt}
\caption{Convergence rate for the $L^2$ projection on the cube minus a sphere geometry.}
\label{tb:proj_4hole3D}
\end{minipage}%
\par
\end{table}

\end{example}

\begin{example}
\emph{Diffusion problem: Dirichlet boundary conditions.}

For each domain we take $f$ so that the solutions $u$ to the problem expressed
in Equation~\eqref{eq:poissong}, with $g=u$ on $\partial\Omega$, are the same
functions $u$ as in the previous example.

We compute the errors in the $H^1(\Oh)$ seminorm and in the $L^2(\Oh)$ norm,
and report the results in Tables~\ref{tb:lap_circle}, \ref{tb:lap_4hole},
\ref{tb:lap_sphere}, and~\ref{tb:lap_4hole3D}, respectively. Also in this case
the barriers stated in Remarks~\ref{re:barrier2} and~\ref{re:barrier3} are
confirmed by our tests.

\begin{table}
\setlength{\tabcolsep}{3pt}
\noindent%
\begin{minipage}[t]{\dimexpr(\linewidth-1\dimexpr0.03\linewidth\relax)/2\relax}
\setlength{\abovecaptionskip}{0pt}\setlength{\belowcaptionskip}{0pt}
\footnotesize
\sbox0{%
\begin{tabular}[t]{rr|cr|cr}
\hline\hline\\[-8pt]
Nelements&Ndofs&\multicolumn{1}{r}{$\|u-u_h\|_{L^2}$}&rate&\multicolumn{1}{r}{$\vert u-u_h\vert_{H^1}$}&rate\\[3pt]
\hline
\multicolumn{6}{c}{$k= 1$}\\
\hline
8424 & 4314 & $1.207\times 10^{-3}$ & -- & $1.102\times 10^{-1}$ & -- \\
33419 & 16912 & $2.989\times 10^{-4}$ & 2.04 & $5.454\times 10^{-2}$ & 1.03 \\
139743 & 70275 & $7.482\times 10^{-5}$ & 1.94 & $2.731\times 10^{-2}$ & 0.97 \\
548805 & 275208 & $1.872\times 10^{-5}$ & 2.03 & $1.366\times 10^{-2}$ & 1.01 \\
2146423 & 1074821 & $4.691\times 10^{-6}$ & 2.03 & $6.845\times 10^{-3}$ & 1.01 \\
8527920 & 4267178 & $1.174\times 10^{-6}$ & 2.01 & $3.425\times 10^{-3}$ & 1.00 \\
\hline
\multicolumn{6}{c}{$k= 2$}\\
\hline
8424 & 17051 & $5.259\times 10^{-4}$ & -- & $8.305\times 10^{-3}$ & -- \\
33419 & 67242 & $1.315\times 10^{-4}$ & 2.02 & $2.973\times 10^{-3}$ & 1.50 \\
139743 & 280292 & $3.288\times 10^{-5}$ & 1.94 & $1.048\times 10^{-3}$ & 1.46 \\
548805 & 1099220 & $8.220\times 10^{-6}$ & 2.03 & $3.729\times 10^{-4}$ & 1.51 \\
2146423 & 4296064 & $2.055\times 10^{-6}$ & 2.03 & $1.315\times 10^{-4}$ & 1.53 \\
8527920 & 17062275 & $5.118\times 10^{-7}$ & 2.02 & $4.650\times 10^{-5}$ & 1.51 \\
\hline
\multicolumn{6}{c}{$k= 3$}\\
\hline
8424 & 38212 & $5.236\times 10^{-4}$ & -- & $6.313\times 10^{-3}$ & -- \\
33419 & 150991 & $1.312\times 10^{-4}$ & 2.01 & $2.249\times 10^{-3}$ & 1.50 \\
139743 & 630052 & $3.284\times 10^{-5}$ & 1.94 & $7.941\times 10^{-4}$ & 1.46 \\
548805 & 2472037 & $8.215\times 10^{-6}$ & 2.03 & $2.816\times 10^{-4}$ & 1.52 \\
2146423 & 9663730 & $2.054\times 10^{-6}$ & 2.03 & $9.949\times 10^{-5}$ & 1.53 \\
8527920 & 38385292 & $5.130\times 10^{-7}$ & 2.01 & $3.517\times 10^{-5}$ & 1.51 \\
\hline
\end{tabular}%
}%
{\centering\ifdim\wd0>\linewidth \resizebox{\linewidth}{!}{\usebox0}\else\usebox0\fi\par}
\par\vspace{10pt}
\caption{Convergence rate for the solution of the Poisson problem on the circle.}
\label{tb:lap_circle}
\end{minipage}%
\hfill
\begin{minipage}[t]{\dimexpr(\linewidth-1\dimexpr0.03\linewidth\relax)/2\relax}
\setlength{\abovecaptionskip}{0pt}\setlength{\belowcaptionskip}{0pt}
\footnotesize
\sbox0{%
\begin{tabular}[t]{rr|cr|cr}
\hline\hline\\[-8pt]
Nelements&Ndofs&\multicolumn{1}{r}{$\|u-u_h\|_{L^2}$}&rate&\multicolumn{1}{r}{$\vert u-u_h\vert_{H^1}$}&rate\\[3pt]
\hline
\multicolumn{6}{c}{$k= 1$}\\
\hline
8785 & 4515 & $5.774\times 10^{-4}$ & -- & $1.090\times 10^{-1}$ & -- \\
35037 & 17762 & $1.435\times 10^{-4}$ & 2.03 & $5.378\times 10^{-2}$ & 1.03 \\
139544 & 70258 & $3.578\times 10^{-5}$ & 2.02 & $2.687\times 10^{-2}$ & 1.01 \\
557133 & 279537 & $8.961\times 10^{-6}$ & 2.01 & $1.346\times 10^{-2}$ & 1.00 \\
2230087 & 1116983 & $2.237\times 10^{-6}$ & 2.00 & $6.723\times 10^{-3}$ & 1.00 \\
8919855 & 4463805 & $5.590\times 10^{-7}$ & 2.00 & $3.361\times 10^{-3}$ & 1.00 \\
\hline
\multicolumn{6}{c}{$k= 2$}\\
\hline
8785 & 17814 & $2.424\times 10^{-5}$ & -- & $1.085\times 10^{-3}$ & -- \\
35037 & 70560 & $6.193\times 10^{-6}$ & 1.98 & $3.641\times 10^{-4}$ & 1.59 \\
139544 & 280059 & $1.550\times 10^{-6}$ & 2.01 & $1.253\times 10^{-4}$ & 1.55 \\
557133 & 1116206 & $3.897\times 10^{-7}$ & 2.00 & $4.333\times 10^{-5}$ & 1.54 \\
2230087 & 4464052 & $9.801\times 10^{-8}$ & 1.99 & $1.526\times 10^{-5}$ & 1.51 \\
8919855 & 17847464 & $2.583\times 10^{-8}$ & 1.92 & $5.379\times 10^{-6}$ & 1.51 \\
\hline
\multicolumn{6}{c}{$k= 3$}\\
\hline
8785 & 39898 & $2.426\times 10^{-5}$ & -- & $7.411\times 10^{-4}$ & -- \\
35037 & 158395 & $6.197\times 10^{-6}$ & 1.98 & $2.620\times 10^{-4}$ & 1.51 \\
139544 & 629404 & $1.551\times 10^{-6}$ & 2.01 & $9.207\times 10^{-5}$ & 1.52 \\
557133 & 2510008 & $3.898\times 10^{-7}$ & 2.00 & $3.240\times 10^{-5}$ & 1.51 \\
2230087 & 10041208 & $9.780\times 10^{-8}$ & 1.99 & $1.146\times 10^{-5}$ & 1.50 \\
8919855 & 40150978 & $2.489\times 10^{-8}$ & 1.97 & $4.052\times 10^{-6}$ & 1.50 \\
\hline
\end{tabular}%
}%
{\centering\ifdim\wd0>\linewidth \resizebox{\linewidth}{!}{\usebox0}\else\usebox0\fi\par}
\par\vspace{10pt}
\caption{Convergence rate for the solution of the Poisson problem on the square minus a circle geometry.}
\label{tb:lap_4hole}
\end{minipage}%
\par
\par\vspace{1.5em}
\noindent%
\begin{minipage}[t]{\dimexpr(\linewidth-1\dimexpr0.03\linewidth\relax)/2\relax}
\setlength{\abovecaptionskip}{0pt}\setlength{\belowcaptionskip}{0pt}
\footnotesize
\sbox0{%
\begin{tabular}[t]{rr|cr|cr}
\hline\hline\\[-8pt]
Nelements&Ndofs&\multicolumn{1}{r}{$\|u-u_h\|_{L^2}$}&rate&\multicolumn{1}{r}{$\vert u-u_h\vert_{H^1}$}&rate\\[3pt]
\hline
\multicolumn{6}{c}{$k= 1$}\\
\hline
5430 & 1259 & $7.591\times 10^{-2}$ & -- & $1.198\times 10^{0}$ & -- \\
40105 & 7825 & $1.968\times 10^{-2}$ & 2.22 & $6.102\times 10^{-1}$ & 1.11 \\
320817 & 56747 & $4.768\times 10^{-3}$ & 2.15 & $2.999\times 10^{-1}$ & 1.08 \\
2473741 & 416162 & $1.152\times 10^{-3}$ & 2.14 & $1.468\times 10^{-1}$ & 1.08 \\
19491401 & 3191093 & $2.831\times 10^{-4}$ & 2.07 & $7.266\times 10^{-2}$ & 1.04 \\
154257125 & 24854019 & $7.092\times 10^{-5}$ & 2.02 & $3.609\times 10^{-2}$ & 1.02 \\
\hline
\multicolumn{6}{c}{$k= 2$}\\
\hline
728 & 1337 & $9.318\times 10^{-2}$ & -- & $5.045\times 10^{-1}$ & -- \\
5430 & 8648 & $2.289\times 10^{-2}$ & 2.26 & $1.798\times 10^{-1}$ & 1.66 \\
40105 & 58448 & $5.905\times 10^{-3}$ & 2.13 & $6.489\times 10^{-2}$ & 1.60 \\
320817 & 445238 & $1.446\times 10^{-3}$ & 2.08 & $2.292\times 10^{-2}$ & 1.54 \\
2473741 & 3349542 & $3.593\times 10^{-4}$ & 2.07 & $8.104\times 10^{-3}$ & 1.55 \\
19491401 & 26047504 & $8.946\times 10^{-5}$ & 2.03 & $2.870\times 10^{-3}$ & 1.52 \\
\hline
\multicolumn{6}{c}{$k= 3$}\\
\hline
271 & 1634 & $1.869\times 10^{-1}$ & -- & $5.750\times 10^{-1}$ & -- \\
1467 & 7948 & $5.371\times 10^{-2}$ & 2.36 & $2.265\times 10^{-1}$ & 1.77 \\
10790 & 53640 & $1.353\times 10^{-2}$ & 2.17 & $8.149\times 10^{-2}$ & 1.61 \\
80074 & 378740 & $3.516\times 10^{-3}$ & 2.07 & $2.981\times 10^{-2}$ & 1.54 \\
622920 & 2867341 & $9.067\times 10^{-4}$ & 2.01 & $1.083\times 10^{-2}$ & 1.50 \\
4936952 & 22416570 & $2.260\times 10^{-4}$ & 2.03 & $3.823\times 10^{-3}$ & 1.52 \\
\hline
\end{tabular}%
}%
{\centering\ifdim\wd0>\linewidth \resizebox{\linewidth}{!}{\usebox0}\else\usebox0\fi\par}
\par\vspace{10pt}
\caption{Convergence rate for the solution of the Poisson problem on the sphere.}
\label{tb:lap_sphere}
\end{minipage}%
\hfill
\begin{minipage}[t]{\dimexpr(\linewidth-1\dimexpr0.03\linewidth\relax)/2\relax}
\setlength{\abovecaptionskip}{0pt}\setlength{\belowcaptionskip}{0pt}
\footnotesize
\sbox0{%
\begin{tabular}[t]{rr|cr|cr}
\hline\hline\\[-8pt]
Nelements&Ndofs&\multicolumn{1}{r}{$\|u-u_h\|_{L^2}$}&rate&\multicolumn{1}{r}{$\vert u-u_h\vert_{H^1}$}&rate\\[3pt]
\hline
\multicolumn{6}{c}{$k= 1$}\\
\hline
5243 & 1321 & $1.301\times 10^{-2}$ & -- & $3.623\times 10^{-1}$ & -- \\
36986 & 7635 & $3.265\times 10^{-3}$ & 2.36 & $1.844\times 10^{-1}$ & 1.16 \\
282071 & 51649 & $8.021\times 10^{-4}$ & 2.20 & $9.235\times 10^{-2}$ & 1.09 \\
2191566 & 374805 & $2.000\times 10^{-4}$ & 2.10 & $4.624\times 10^{-2}$ & 1.05 \\
17352530 & 2864003 & $4.948\times 10^{-5}$ & 2.06 & $2.307\times 10^{-2}$ & 1.03 \\
138171918 & 22340094 & $1.248\times 10^{-5}$ & 2.01 & $1.150\times 10^{-2}$ & 1.02 \\
\hline
\multicolumn{6}{c}{$k= 2$}\\
\hline
1012 & 1935 & $2.271\times 10^{-3}$ & -- & $5.184\times 10^{-2}$ & -- \\
5243 & 8756 & $5.550\times 10^{-4}$ & 2.80 & $1.557\times 10^{-2}$ & 2.39 \\
36986 & 55535 & $1.303\times 10^{-4}$ & 2.35 & $4.759\times 10^{-3}$ & 1.93 \\
282071 & 398290 & $3.182\times 10^{-5}$ & 2.15 & $1.473\times 10^{-3}$ & 1.79 \\
2191566 & 2991492 & $7.988\times 10^{-6}$ & 2.06 & $4.855\times 10^{-4}$ & 1.65 \\
17352530 & 23280977 & $2.061\times 10^{-6}$ & 1.98 & $1.665\times 10^{-4}$ & 1.56 \\
\hline
\multicolumn{6}{c}{$k= 3$}\\
\hline
364 & 2210 & $3.099\times 10^{-3}$ & -- & $4.743\times 10^{-2}$ & -- \\
1544 & 8652 & $1.147\times 10^{-3}$ & 2.18 & $1.982\times 10^{-2}$ & 1.92 \\
9665 & 49272 & $3.344\times 10^{-4}$ & 2.13 & $6.702\times 10^{-3}$ & 1.87 \\
72571 & 348698 & $8.677\times 10^{-5}$ & 2.07 & $2.138\times 10^{-3}$ & 1.75 \\
554579 & 2573625 & $2.027\times 10^{-5}$ & 2.18 & $6.706\times 10^{-4}$ & 1.74 \\
4366891 & 19913036 & $5.128\times 10^{-6}$ & 2.02 & $2.275\times 10^{-4}$ & 1.59 \\
\hline
\end{tabular}%
}%
{\centering\ifdim\wd0>\linewidth \resizebox{\linewidth}{!}{\usebox0}\else\usebox0\fi\par}
\par\vspace{10pt}
\caption{Convergence rate for the solution of the Poisson problem on the cube minus a sphere geometry.}
\label{tb:lap_4hole3D}
\end{minipage}%
\par
\end{table}

\end{example}

\begin{example}
\emph{(Reaction-)diffusion problem: Neumann boundary conditions.}

We repeat the same set of computations for the variational problem described
in Section~\ref{se:neumann}. We enforce homogeneous Neumann boundary
conditions on the curved parts of the boundary and Dirichlet boundary
conditions on the remaining parts (matching the exact solution $u$). The data
are chosen so that the exact solutions $u$ read as follows:
\[
\aligned
&u(x,y)=(x+y)e^{-(x^2+y^2)/2}&&\text{for the circle}\\
&u(x,y)=xy(x^2+y^2-1/2)&&\text{for the square minus the circle}\\
&u(x,y,z)=xe^{-(x^2+y^2+z^2)/2}&&\text{for the sphere}\\
&u(x,y,z)=xyz(x^2+y^2+z^2-5/12)&&\text{for the cube minus the sphere}
\endaligned
\]
In all cases $u$ has vanishing normal derivative along the curved parts of the
boundary of $\Omega$. When mixed boundary conditions are considered, only the
diffusion part of the equation is solved since the reaction term is not needed
for the well-posedness.

We compute the errors in the $H^1(\Oh)$ seminorm and in the $L^2(\Oh)$ norm,
and report the results in Tables~\ref{tb:lapN_circle}, \ref{tb:lapN_4hole},
\ref{tb:lapN_sphere}, and~\ref{tb:lapN_4hole3D}, respectively.

\begin{table}
\setlength{\tabcolsep}{3pt}
\noindent%
\begin{minipage}[t]{\dimexpr(\linewidth-1\dimexpr0.03\linewidth\relax)/2\relax}
\setlength{\abovecaptionskip}{0pt}\setlength{\belowcaptionskip}{0pt}
\footnotesize
\sbox0{%
\begin{tabular}[t]{rr|cr|cr}
\hline\hline\\[-8pt]
Nelements&Ndofs&\multicolumn{1}{r}{$\|u-u_h\|_{L^2}$}&rate&\multicolumn{1}{r}{$\vert u-u_h\vert_{H^1}$}&rate\\[3pt]
\hline
\multicolumn{6}{c}{$k= 1$}\\
\hline
8424 & 4314 & $2.066\times 10^{-4}$ & -- & $1.984\times 10^{-2}$ & -- \\
33419 & 16912 & $5.184\times 10^{-5}$ & 2.02 & $9.930\times 10^{-3}$ & 1.01 \\
139743 & 70275 & $1.293\times 10^{-5}$ & 1.95 & $4.951\times 10^{-3}$ & 0.98 \\
548805 & 275208 & $3.228\times 10^{-6}$ & 2.03 & $2.472\times 10^{-3}$ & 1.02 \\
2146423 & 1074821 & $8.097\times 10^{-7}$ & 2.03 & $1.240\times 10^{-3}$ & 1.01 \\
8527920 & 4267178 & $2.029\times 10^{-7}$ & 2.01 & $6.209\times 10^{-4}$ & 1.00 \\
\hline
\multicolumn{6}{c}{$k= 2$}\\
\hline
8424 & 17051 & $9.500\times 10^{-5}$ & -- & $2.456\times 10^{-4}$ & -- \\
33419 & 67242 & $2.387\times 10^{-5}$ & 2.01 & $6.251\times 10^{-5}$ & 1.99 \\
139743 & 280292 & $5.982\times 10^{-6}$ & 1.94 & $1.505\times 10^{-5}$ & 1.99 \\
548805 & 1099220 & $1.497\times 10^{-6}$ & 2.03 & $3.788\times 10^{-6}$ & 2.02 \\
2146423 & 4296064 & $3.742\times 10^{-7}$ & 2.03 & $9.613\times 10^{-7}$ & 2.01 \\
8527920 & 17062275 & $9.219\times 10^{-8}$ & 2.03 & $2.430\times 10^{-7}$ & 1.99 \\
\hline
\multicolumn{6}{c}{$k= 3$}\\
\hline
8424 & 38212 & $9.491\times 10^{-5}$ & -- & $7.120\times 10^{-4}$ & -- \\
33419 & 150991 & $2.386\times 10^{-5}$ & 2.01 & $2.468\times 10^{-4}$ & 1.54 \\
139743 & 630052 & $5.980\times 10^{-6}$ & 1.94 & $8.689\times 10^{-5}$ & 1.46 \\
548805 & 2472037 & $1.496\times 10^{-6}$ & 2.03 & $3.050\times 10^{-5}$ & 1.53 \\
2146423 & 9663730 & $3.687\times 10^{-7}$ & 2.05 & $1.077\times 10^{-5}$ & 1.53 \\
8527920 & 38385292 & $7.042\times 10^{-8}$ & 2.40 & $3.798\times 10^{-6}$ & 1.51 \\
\hline
\end{tabular}%
}%
{\centering\ifdim\wd0>\linewidth \resizebox{\linewidth}{!}{\usebox0}\else\usebox0\fi\par}
\par\vspace{10pt}
\caption{Convergence rate for the solution of the reaction-diffusion problem on the circle with Neumann boundary conditions.}
\label{tb:lapN_circle}
\end{minipage}%
\hfill
\begin{minipage}[t]{\dimexpr(\linewidth-1\dimexpr0.03\linewidth\relax)/2\relax}
\setlength{\abovecaptionskip}{0pt}\setlength{\belowcaptionskip}{0pt}
\footnotesize
\sbox0{%
\begin{tabular}[t]{rr|cr|cr}
\hline\hline\\[-8pt]
Nelements&Ndofs&\multicolumn{1}{r}{$\|u-u_h\|_{L^2}$}&rate&\multicolumn{1}{r}{$\vert u-u_h\vert_{H^1}$}&rate\\[3pt]
\hline
\multicolumn{6}{c}{$k= 1$}\\
\hline
8785 & 4515 & $6.798\times 10^{-5}$ & -- & $1.542\times 10^{-2}$ & -- \\
35037 & 17762 & $1.687\times 10^{-5}$ & 2.03 & $7.633\times 10^{-3}$ & 1.03 \\
139544 & 70258 & $4.210\times 10^{-6}$ & 2.02 & $3.817\times 10^{-3}$ & 1.01 \\
557133 & 279537 & $1.055\times 10^{-6}$ & 2.00 & $1.912\times 10^{-3}$ & 1.00 \\
2230087 & 1116983 & $2.633\times 10^{-7}$ & 2.00 & $9.549\times 10^{-4}$ & 1.00 \\
8919855 & 4463805 & $6.579\times 10^{-8}$ & 2.00 & $4.774\times 10^{-4}$ & 1.00 \\
\hline
\multicolumn{6}{c}{$k= 2$}\\
\hline
8785 & 17814 & $2.246\times 10^{-6}$ & -- & $6.651\times 10^{-5}$ & -- \\
35037 & 70560 & $5.718\times 10^{-7}$ & 1.99 & $1.650\times 10^{-5}$ & 2.03 \\
139544 & 280059 & $1.429\times 10^{-7}$ & 2.01 & $4.125\times 10^{-6}$ & 2.01 \\
557133 & 1116206 & $3.589\times 10^{-8}$ & 2.00 & $1.035\times 10^{-6}$ & 2.00 \\
2230087 & 4464052 & $8.968\times 10^{-9}$ & 2.00 & $2.584\times 10^{-7}$ & 2.00 \\
8919855 & 17847464 & $2.142\times 10^{-9}$ & 2.07 & $6.454\times 10^{-8}$ & 2.00 \\
\hline
\multicolumn{6}{c}{$k= 3$}\\
\hline
8785 & 39898 & $2.244\times 10^{-6}$ & -- & $2.735\times 10^{-5}$ & -- \\
35037 & 158395 & $5.719\times 10^{-7}$ & 1.98 & $9.368\times 10^{-6}$ & 1.55 \\
139544 & 629404 & $1.429\times 10^{-7}$ & 2.01 & $3.197\times 10^{-6}$ & 1.56 \\
557133 & 2510008 & $3.590\times 10^{-8}$ & 2.00 & $1.124\times 10^{-6}$ & 1.51 \\
2230087 & 10041208 & $8.989\times 10^{-9}$ & 2.00 & $3.937\times 10^{-7}$ & 1.51 \\
8919855 & 40150978 & $2.215\times 10^{-9}$ & 2.02 & $1.385\times 10^{-7}$ & 1.51 \\
\hline
\end{tabular}%
}%
{\centering\ifdim\wd0>\linewidth \resizebox{\linewidth}{!}{\usebox0}\else\usebox0\fi\par}
\par\vspace{10pt}
\caption{Convergence rate for the solution of the Poisson problem on the square minus a circle geometry with mixed boundary conditions.}
\label{tb:lapN_4hole}
\end{minipage}%
\par
\par\vspace{1.5em}
\noindent%
\begin{minipage}[t]{\dimexpr(\linewidth-1\dimexpr0.03\linewidth\relax)/2\relax}
\setlength{\abovecaptionskip}{0pt}\setlength{\belowcaptionskip}{0pt}
\footnotesize
\sbox0{%
\begin{tabular}[t]{rr|cr|cr}
\hline\hline\\[-8pt]
Nelements&Ndofs&\multicolumn{1}{r}{$\|u-u_h\|_{L^2}$}&rate&\multicolumn{1}{r}{$\vert u-u_h\vert_{H^1}$}&rate\\[3pt]
\hline
\multicolumn{6}{c}{$k= 1$}\\
\hline
5430 & 1259 & $7.208\times 10^{-3}$ & -- & $1.169\times 10^{-1}$ & -- \\
40105 & 7825 & $1.773\times 10^{-3}$ & 2.30 & $5.779\times 10^{-2}$ & 1.16 \\
320817 & 56747 & $4.375\times 10^{-4}$ & 2.12 & $2.850\times 10^{-2}$ & 1.07 \\
2473741 & 416162 & $1.085\times 10^{-4}$ & 2.10 & $1.421\times 10^{-2}$ & 1.05 \\
19491401 & 3191093 & $2.708\times 10^{-5}$ & 2.04 & $7.095\times 10^{-3}$ & 1.02 \\
154257125 & 24854019 & $6.892\times 10^{-6}$ & 2.00 & $3.545\times 10^{-3}$ & 1.01 \\
\hline
\multicolumn{6}{c}{$k= 2$}\\
\hline
728 & 1337 & $1.089\times 10^{-2}$ & -- & $3.127\times 10^{-2}$ & -- \\
5430 & 8648 & $2.627\times 10^{-3}$ & 2.29 & $8.318\times 10^{-3}$ & 2.13 \\
40105 & 58448 & $6.601\times 10^{-4}$ & 2.17 & $2.228\times 10^{-3}$ & 2.07 \\
320817 & 445238 & $1.677\times 10^{-4}$ & 2.02 & $6.064\times 10^{-4}$ & 1.92 \\
2473741 & 3349542 & $4.114\times 10^{-5}$ & 2.09 & $1.720\times 10^{-4}$ & 1.87 \\
19491401 & 26047504 & $1.024\times 10^{-5}$ & 2.03 & $5.182\times 10^{-5}$ & 1.75 \\
\hline
\multicolumn{6}{c}{$k= 3$}\\
\hline
271 & 1634 & $2.264\times 10^{-2}$ & -- & $6.147\times 10^{-2}$ & -- \\
1467 & 7948 & $6.109\times 10^{-3}$ & 2.48 & $1.794\times 10^{-2}$ & 2.34 \\
10790 & 53640 & $1.583\times 10^{-3}$ & 2.12 & $5.333\times 10^{-3}$ & 1.91 \\
80074 & 378740 & $4.123\times 10^{-4}$ & 2.06 & $1.700\times 10^{-3}$ & 1.75 \\
622920 & 2867341 & $1.040\times 10^{-4}$ & 2.04 & $5.609\times 10^{-4}$ & 1.64 \\
4936952 & 22416570 & $2.608\times 10^{-5}$ & 2.02 & $1.889\times 10^{-4}$ & 1.59 \\
\hline
\end{tabular}%
}%
{\centering\ifdim\wd0>\linewidth \resizebox{\linewidth}{!}{\usebox0}\else\usebox0\fi\par}
\par\vspace{10pt}
\caption{Convergence rate for the solution of the reaction-diffusion problem on the sphere with Neumann boundary conditions.}
\label{tb:lapN_sphere}
\end{minipage}%
\hfill
\begin{minipage}[t]{\dimexpr(\linewidth-1\dimexpr0.03\linewidth\relax)/2\relax}
\setlength{\abovecaptionskip}{0pt}\setlength{\belowcaptionskip}{0pt}
\footnotesize
\sbox0{%
\begin{tabular}[t]{rr|cr|cr}
\hline\hline\\[-8pt]
Nelements&Ndofs&\multicolumn{1}{r}{$\|u-u_h\|_{L^2}$}&rate&\multicolumn{1}{r}{$\vert u-u_h\vert_{H^1}$}&rate\\[3pt]
\hline
\multicolumn{6}{c}{$k= 1$}\\
\hline
5243 & 1321 & $3.241\times 10^{-3}$ & -- & $1.478\times 10^{-1}$ & -- \\
36986 & 7635 & $8.094\times 10^{-4}$ & 2.37 & $7.618\times 10^{-2}$ & 1.13 \\
282071 & 51649 & $1.991\times 10^{-4}$ & 2.20 & $3.834\times 10^{-2}$ & 1.08 \\
2191566 & 374805 & $4.966\times 10^{-5}$ & 2.10 & $1.906\times 10^{-2}$ & 1.06 \\
17352530 & 2864003 & $1.233\times 10^{-5}$ & 2.06 & $9.480\times 10^{-3}$ & 1.03 \\
138171918 & 22340094 & $3.106\times 10^{-6}$ & 2.01 & $4.724\times 10^{-3}$ & 1.02 \\
\hline
\multicolumn{6}{c}{$k= 2$}\\
\hline
1012 & 1935 & $4.122\times 10^{-4}$ & -- & $1.639\times 10^{-2}$ & -- \\
5243 & 8756 & $7.171\times 10^{-5}$ & 3.48 & $5.300\times 10^{-3}$ & 2.24 \\
36986 & 55535 & $9.753\times 10^{-6}$ & 3.24 & $1.385\times 10^{-3}$ & 2.18 \\
282071 & 398290 & $1.305\times 10^{-6}$ & 3.06 & $3.504\times 10^{-4}$ & 2.09 \\
2191566 & 2991492 & $1.891\times 10^{-7}$ & 2.87 & $8.766\times 10^{-5}$ & 2.06 \\
17352530 & 23280977 & $3.466\times 10^{-8}$ & 2.48 & $2.187\times 10^{-5}$ & 2.03 \\
\hline
\multicolumn{6}{c}{$k= 3$}\\
\hline
364 & 2210 & $6.648\times 10^{-5}$ & -- & $2.233\times 10^{-3}$ & -- \\
1544 & 8652 & $1.573\times 10^{-5}$ & 3.17 & $4.598\times 10^{-4}$ & 3.47 \\
9665 & 49272 & $4.684\times 10^{-6}$ & 2.09 & $7.624\times 10^{-5}$ & 3.10 \\
72571 & 348698 & $1.256\times 10^{-6}$ & 2.02 & $1.562\times 10^{-5}$ & 2.43 \\
554579 & 2573625 & $2.854\times 10^{-7}$ & 2.22 & $3.760\times 10^{-6}$ & 2.14 \\
4366891 & 19913036 & $7.046\times 10^{-8}$ & 2.05 & $1.129\times 10^{-6}$ & 1.76 \\
\hline
\end{tabular}%
}%
{\centering\ifdim\wd0>\linewidth \resizebox{\linewidth}{!}{\usebox0}\else\usebox0\fi\par}
\par\vspace{10pt}
\caption{Convergence rate for the solution of the Poisson problem on the cube minus a sphere geometry with mixed boundary conditions.}
\label{tb:lapN_4hole3D}
\end{minipage}%
\par
\end{table}

The results confirm the theoretical barriers recalled in
Remark~\ref{re:barrier4} with three apparent exceptions. Namely, both
two-dimensional cases have a superconvergent behavior for the $H^1$ seminorm
when $k=2$ is used. While the theory predicts $O(h^{3/2})$, the numerical
tests show $O(h^2)$.
Similarly, the last three-dimensional test, shown in
Table~\ref{tb:lapN_4hole3D}, has a superconvergent behavior when $k=2$ is
used.

The behavior in the two-dimensional cases can be explained theoretically and
is due to a lucky cancellation of terms arising from some symmetries of
polynomials of degree two and the regions $\Smt$ and $\Spt$ in our domains.

The superconvergence of Table~\ref{tb:lapN_4hole3D}, however, has to be
interpreted as a pre-asymptotic regime. There are two reasons for this: on the
one hand, the curved part of the boundary is not a significant portion of it;
on the other hand, the chosen exact solution has some features that reduce the
impact of the geometry mismatch. Table~\ref{tb:lapN_4hole3D_R0p75} shows the
results of a similar computation where a larger sphere (radius $3/4$ instead
of $1/2$) is subtracted from the unit cube. The exact solution in this case is
\[
u(x,y,z)=xyz(x^2+y^2+z^2-15/16).
\]
The first phenomenon is alleviated, while the second one is still present and
we can see that the $L^2$ convergence is restored to the capped second order,
while the $H^1$ error is still superconverging.

\begin{table}
\setlength{\tabcolsep}{3pt}
\noindent\hspace*{\fill}%
\begin{minipage}[t]{\dimexpr(\linewidth-1\dimexpr0.03\linewidth\relax)/2\relax}
\setlength{\abovecaptionskip}{0pt}\setlength{\belowcaptionskip}{0pt}
\footnotesize
\sbox0{%
\begin{tabular}[t]{rr|cr|cr}
\hline\hline\\[-8pt]
Nelements&Ndofs&\multicolumn{1}{r}{$\|u-u_h\|_{L^2}$}&rate&\multicolumn{1}{r}{$\vert u-u_h\vert_{H^1}$}&rate\\[3pt]
\hline
\multicolumn{6}{c}{$k= 1$}\\
\hline
2340 & 648 & $4.815\times 10^{-3}$ & -- & $1.518\times 10^{-1}$ & -- \\
15730 & 3547 & $1.247\times 10^{-3}$ & 2.38 & $8.386\times 10^{-2}$ & 1.05 \\
118434 & 22731 & $3.022\times 10^{-4}$ & 2.29 & $4.133\times 10^{-2}$ & 1.14 \\
921457 & 162108 & $7.306\times 10^{-5}$ & 2.17 & $2.043\times 10^{-2}$ & 1.08 \\
7266769 & 1218299 & $1.793\times 10^{-5}$ & 2.09 & $1.016\times 10^{-2}$ & 1.04 \\
57562510 & 9373960 & $4.522\times 10^{-6}$ & 2.03 & $5.061\times 10^{-3}$ & 1.03 \\
\hline
\multicolumn{6}{c}{$k= 2$}\\
\hline
887 & 1742 & $4.109\times 10^{-4}$ & -- & $1.498\times 10^{-2}$ & -- \\
4511 & 7656 & $7.358\times 10^{-5}$ & 3.49 & $4.809\times 10^{-3}$ & 2.30 \\
31038 & 47269 & $1.324\times 10^{-5}$ & 2.83 & $1.312\times 10^{-3}$ & 2.14 \\
236633 & 337124 & $2.792\times 10^{-6}$ & 2.38 & $3.289\times 10^{-4}$ & 2.11 \\
1832322 & 2512610 & $6.359\times 10^{-7}$ & 2.21 & $8.234\times 10^{-5}$ & 2.07 \\
14488188 & 19488116 & $1.516\times 10^{-7}$ & 2.10 & $2.057\times 10^{-5}$ & 2.03 \\
\hline
\multicolumn{6}{c}{$k= 3$}\\
\hline
293 & 1842 & $3.227\times 10^{-4}$ & -- & $3.329\times 10^{-3}$ & -- \\
1301 & 7458 & $9.262\times 10^{-5}$ & 2.68 & $8.203\times 10^{-4}$ & 3.01 \\
8258 & 42765 & $2.246\times 10^{-5}$ & 2.43 & $1.975\times 10^{-4}$ & 2.45 \\
61065 & 296185 & $5.404\times 10^{-6}$ & 2.21 & $5.359\times 10^{-5}$ & 2.02 \\
465680 & 2171883 & $1.440\times 10^{-6}$ & 1.99 & $1.677\times 10^{-5}$ & 1.75 \\
3647187 & 16674769 & $3.892\times 10^{-7}$ & 1.93 & $5.375\times 10^{-6}$ & 1.67 \\
\hline
\end{tabular}%
}%
{\centering\ifdim\wd0>\linewidth \resizebox{\linewidth}{!}{\usebox0}\else\usebox0\fi\par}
\par\vspace{10pt}
\caption{Convergence rate for the solution of the Poisson problem on the cube minus a sphere of radius $R=0.75$ geometry with mixed boundary conditions.}
\label{tb:lapN_4hole3D_R0p75}
\end{minipage}%
\hspace*{\fill}\par
\end{table}

This phenomenon will be further investigated in a future project.

\end{example}

\section*{AI disclosure}
\label{se:AI}
The authors used AI tools for several tasks. They drafted the computational
codes and then asked first Copilot and then, more extensively, Claude to create an
organic integrated package (available on request). Different models of Claude
were used, ranging from Opus 4.7 to the latest Fable 5.1. 
While the authors are fully responsible for the content of this manuscript, on
several occasions they shared ideas and checked some results with the AI tool.
Claude, in particular, pointed out to the authors the need to use $\Spt$ and
$\Smt$ instead of $\Tt\setminus T$ and $T\setminus\Tt$, respectively, in three
dimensions.
The same AI tool was used to check the grammar and typos, as well as to
generate figures and tables.

\section*{Acknowledgments}

DB is a member of the INdAM research group GNCS and his work was partially 
supported by KAUST -- CRG13-2025 grant (6911). LG is a member of the INdAM 
research group GNCS.

\bibliographystyle{siamplain}
\bibliography{ref}

\end{document}